\documentclass[11pt,a4paper]{amsart}
\usepackage[T1]{fontenc}
\usepackage{lmodern}
\usepackage[top=2.5cm,bottom=2.5cm,left=2.5cm,right=2.5cm]{geometry}
\usepackage{amsmath,amssymb,amsthm,mathtools,mathrsfs}
\usepackage{microtype}
\usepackage{xcolor}
\usepackage{hyperref}
\hypersetup{colorlinks=true,linkcolor=blue!60!black,
 citecolor=green!45!black,urlcolor=green!45!black}
\numberwithin{equation}{section}
\theoremstyle{definition}
\newtheorem{theorem}{Theorem}[section]
\newtheorem{proposition}[theorem]{Proposition}
\newtheorem{lemma}[theorem]{Lemma}
\newtheorem{corollary}[theorem]{Corollary}
\newtheorem{remark}[theorem]{Remark}
\newtheorem{example}[theorem]{Example}
\newcommand{\PP}{\mathcal{P}}
\newcommand{\R}{\mathbb R}
\newcommand{\N}{\mathbb N}
\newcommand{\Hh}{\mathcal H}
\newcommand{\HH}{\mathcal H}
\newcommand{\LL}{\mathcal L}

\newcommand{\E}{\mathbb E}
\newcommand{\U}{\mathcal U}
\newcommand{\weak}{\rightharpoonup}
\newcommand{\res}{\mathbin{\vrule height 1.6ex depth 0pt width
0.13ex\vrule height 0.13ex depth 0pt width 1.3ex}}

\newcommand{\ip}[2]{\langle #1,#2\rangle}
\newcommand{\norm}[1]{\lVert #1\rVert}
\DeclareMathOperator{\supp}{spt}
\DeclareMathOperator{\dist}{dist}

\DeclareMathOperator{\spanop}{span}
\DeclareMathOperator{\tr}{tr}

\newcommand{\eps}{\varepsilon}
\renewcommand{\epsilon}{\varepsilon}
\renewcommand{\bar}{\overline}
\allowdisplaybreaks[2]
\title[Gaussian rigidity and width of measures]
 {Gaussian rigidity and width of measures,\\
 with an application to Preiss's theorem}
\date{}
\author[A. Pigati]{Alessandro Pigati}
\address{Department of Decision Sciences, Bocconi University, Via Roentgen 1, 20136 Milano, Italy.}
\email{alessandro.pigati@unibocconi.it}

\begin{document}
\begin{abstract}
Given $0\le m<n$ and a positive Radon measure $\mu$ on $\R^n$ with at most polynomial growth,
assume that it has finite $m$-dimensional \emph{Gaussian discrepancy}, defined as
$$\mathcal D(\mu):=\limsup_{R\to\infty}\Bigl[R^2\sup_{x\in\supp\mu}|G_R^\mu(x)-1|\Bigr],$$
where $G_R^\mu(x):=(4\pi R^2)^{-m/2}\int_{\R^n} e^{-|z-x|^2/(4R^2)}\,d\mu(z)$.
Provided that $\mu$ admits a flat blow-down $\HH^m\res V$ (an assumption that cannot be dropped entirely),
we prove that there exists a finite (smallest) $0\le\delta<\infty$, called the \emph{width} of $\mu$, such that
$$\supp\mu\subseteq \bar B_\delta(V+x_0),\quad \delta^2\le C(m,n)\mathcal{D}(\mu),$$
where $\bar B_\delta(V+x_0)$ denotes the closed $\delta$-neighborhood of a suitable translate of $V$.
Note that the exponent is optimal because of scale invariance.

In particular, when $\mathcal D(\mu)=0$, the support of $\mu$ is included in an $m$-dimensional affine plane.
We use this fact to give a simpler proof of Preiss's rectifiability theorem \cite[Theorem~5.6]{Preiss}:
while still based on Gaussian averages of $\mu$, this new proof avoids showing uniqueness (and hence conicality) of blow-downs of $m$-uniform measures, as well as studying Gaussian-weighted moments beyond second order.

The strategy of proof consists of an anisotropic blow-down of $\mu$, followed by an analysis of the resulting Young measure, using a rigidity lemma involving the heat semigroup.
\end{abstract}
\maketitle

\section{Introduction}

Given an integer $0\le m<n$ and a positive Radon measure $\mu$ on
$\R^n$, we consider its Gaussian averages $G_R^\mu(x)$, given at each center $x\in\R^n$ and scale $R>0$ by
\begin{equation}\label{eq:gaussian-measure}
 G_R^\mu(x):=(4\pi R^2)^{-m/2} \int_{\R^n} e^{-|z-x|^2/(4R^2)}\,d\mu(z),
\end{equation}
and assume that $G_R^\mu(x)<\infty$ for all $x\in\supp\mu$ and $R>0$.
The normalization factor gives $G_R^\mu(x)=1$ for all $x\in\supp\mu$ if $\mu$ is \emph{flat},
namely $\mu=\HH^m\res V$ for a suitable affine $m$-plane $V$,
or more generally if $\mu$ is \emph{$m$-uniform}, meaning that $\mu(B_r(x))=\omega_mr^m$ for all $x\in\supp\mu$ and $r>0$.

We then define the \emph{Gaussian discrepancy} to be
\begin{equation}\label{eq:defect-infinity}
 \mathcal D(\mu):=\limsup_{R\to\infty} R^2\sup_{x\in\supp\mu}|G_R^\mu(x)-1|.
\end{equation}

Using the notation
\[
 S_{x,r}\mu:=r^{-m}(z\mapsto(z-x)/r)_*\mu,
 \quad S_r\mu:=S_{0,r}\mu
\]
for the rescalings of $\mu$, recall that a \emph{blow-down} is a weak-* limit $\lim_{j\to\infty}S_{R_j}\mu$,
along a sequence of scales $R_j\to\infty$ (as usual, weak-* convergence is understood in duality with $C^0_c(\R^n)$).

In this article, we first prove the following rigidity result.
\begin{theorem}\label{thm:zero}
If $\mu$ has a flat blow-down $\Hh^m\res V$ and $\mathcal{D}(\mu)=0$, namely
\begin{equation}\label{eq:zero-defect}
 \sup_{x\in\supp\mu}|G_R^\mu(x)-1|=o(R^{-2})
 \quad\text{as }R\to\infty,
\end{equation}
then $\supp\mu$ is contained in an affine $m$-plane parallel to $V$.
\end{theorem}

\begin{remark}
As a preliminary step in the proof (Lemma \ref{lem:coarse}) will show, we also have
$R_j^{-1}\supp\mu\to V$ in the local Hausdorff sense, and thus $0\in V$.
The assumption of a flat blow-down is necessary, at least in arbitrary dimension: in $\R^4$, the Kowalski--Preiss cone
$$\mu:=\Hh^3\res\mathcal{C},\quad \mathcal{C}:=\{x\in\R^4:x_4^2=x_1^2+x_2^2+x_3^2\}$$
is a $3$-uniform measure \cite[Proposition~3.8]{KP}, and hence has $\mathcal{D}(\mu)=0$.
\end{remark}

\begin{example}
Note that Theorem \ref{thm:zero} can also be applied to lower-dimensional measures, such as atomic measures:
for instance, letting $\mu:=\sum_{a\in\mathbb Z^m\times\{0\}}\delta_a$, it is easy to check that $\mathcal{D}(\mu)=0$,
and indeed $\supp\mu\subset\R^m\times\{0\}$.
\end{example}

Generalizing Theorem \ref{thm:zero}, for an $m$-plane $V\ni0$, we let
\begin{equation}\label{eq:width}
 w_V(\mu):=\inf_{c\in V^\perp}
       \sup_{x\in\supp\mu}|\pi_{V^\perp}x-c|,
\end{equation}
which we call the \emph{width} of $\mu$. The following quantitative version of Theorem \ref{thm:zero} holds.

\begin{theorem}\label{thm:width}
For every $0\le m<n$ there exists $C(m,n)>0$ such that, whenever $\mu$ has a flat blow-down $\Hh^m\res V$ and
$\mathcal D(\mu)<\infty$, we have
\begin{equation}\label{eq:main-width}
 w_V(\mu)^2\le C\mathcal D(\mu).
\end{equation}
\end{theorem}

\begin{remark}
While the statement is trivial when $m=n$, the reader might wonder what can be said when $m\in[0,\infty)\setminus\{0,1,\dots,n\}$. In fact, in this case the assumption $\mathcal{D}(\mu)<\infty$ implies $\mu=0$: if we had $\mu\neq0$, then any blow-down of $\mu$ would be a nontrivial $m$-uniform measure (as shown by the proof of Lemma \ref{lem:coarse}), which is impossible since such measures do not exist (see, e.g., the proof of Proposition \ref{lem:flat-tangent}).
\end{remark}

\begin{example}\label{ex:two-planes}
Given an $m$-plane $V\ni0$ and $a\in V^\perp\setminus\{0\}$, consider $\mu:=\frac12\Hh^m\res(V+a)+\frac12\Hh^m\res(V-a)$, for which the blow-down is $\HH^m\res V$. Then $w_V(\mu)=|a|$, while at every $x\in\supp\mu$ we have
\[
 G_R^\mu(x)=\frac{1+e^{-|a|^2/R^2}}2=1-\frac{|a|^2}{2R^2}+O_a(R^{-4}),
\]
so that $\mathcal D(\mu)=|a|^2/2$.
\end{example}

\begin{example}\label{ex:eight}
In $\R^3$, for $m=2$, let $V:=\R^2\times\{0\}$ and
\[
 \mu:=\Hh^2\res V+\gamma\delta_{\lambda e_3}+\gamma\delta_{-\lambda e_3},
\]
for arbitrary $\gamma,\lambda>0$. Its blow-down is $\Hh^2\res V$ and its width is $w_V(\mu)=\lambda$.
For any $p\in\R^2$ we can compute that
\[
 G_R^\mu(p,0)=1+\frac{\gamma}{2\pi R^2}e^{-(|p|^2+\lambda^2)/(4R^2)},
\]
while
\[
 G_R^\mu(\pm\lambda e_3)
 =e^{-\lambda^2/(4R^2)}+\frac{\gamma}{4\pi R^2}(1+e^{-\lambda^2/R^2})
 =1-\frac{(\lambda^2/4)-(\gamma/(2\pi))}{R^2}+O_{\gamma,\lambda}(R^{-4}).
\]
Consequently, we have $\mathcal D(\mu)=\max\{\frac\gamma{2\pi},|\frac{\lambda^2}{4}-\frac\gamma{2\pi}|\}$.
In particular, taking $\gamma:=\pi\lambda^2/4$ gives $\mathcal D(\mu)=\lambda^2/8$, showing that $C(m,n)\ge8$ for every $m\ge2$ and $n\ge m+1$.
\end{example}

As already mentioned, Theorem \ref{thm:zero} applies to $m$-uniform measures, in which case it gives
the classical assertion that an $m$-uniform measure with a flat blow-down is flat \cite[Corollary~3.16]{Preiss}.
This constitutes the main rigidity step in the proof of the following rectifiability theorem of Preiss~\cite[Theorem~5.6]{Preiss}. The original proof of this step was perhaps the most technical part of the proof of Preiss's result, involving a careful analysis of Gaussian-weighted moments of $\mu$ up to order four.

\begin{theorem}[Preiss]\label{thm:preiss}
Let $\mu\ne0$ be a positive Radon measure on $\R^n$ and let $\alpha\ge0$.
Assume that $\Theta^\alpha(\mu,x):=\lim_{r\to0}\frac{\mu(B_r(x))}{\omega_\alpha r^\alpha}$ exists
and belongs to $(0,\infty)$, at $\mu$-a.e.\ $x$.
Then $\alpha=m\in\{0,\ldots,n\}$ and $\mu$ is $m$-rectifiable.
\end{theorem}

The aforementioned assertion was also a key tool in the proof of other results, such as those in \cite{Tolsa}.
Theorem \ref{thm:preiss} was recently extended to codimension-one measures in Heisenberg groups in \cite{Merlo}.
In fact, with slightly more work, one can also get the following stronger statement, which first appeared in \cite[Corollary~5.5]{Preiss}.

\begin{theorem}[Preiss]\label{preiss.bis}
The same holds if, rather than assuming the density $\Theta^\alpha(\mu,x)$ to exist,
we just assume
\begin{equation}\label{eq:preiss.bis}
0<\Theta^{\alpha,*}(\mu,x)\le(1+\epsilon_0)\Theta_*^\alpha(\mu,x)<\infty
 \quad\text{for }\mu\text{-a.e. }x,
\end{equation}
for a suitable threshold $\epsilon_0(\alpha,n)>0$.
\end{theorem}

It should be mentioned that the use of Gaussian averages to study measures already appeared several times,
particularly in the study of uniform measures (see, e.g., \cite{KiP,Preiss}).

\subsection{Strategy of proof}

Besides coarse growth bounds, in Section \ref{sec:prel} we show that
$\pi_{V^\perp}\supp\mu$ is bounded whenever $\mathcal{D}(\mu)<\infty$ and $\mu$ has a flat blow-down $\HH^m\res V$,
by means of a simple Taylor expansion. Further, in the case $\mathcal{D}(\mu)=0$,
we also show that $\supp\mu$ is contained in the shell $\{x\in\R^n:|\pi_{V^\perp}x|=|b|\}$
(for some $b\in V^\perp$, up to a translation such that $-b\in\supp\mu$).

In Section \ref{sec:function} we prove Theorem \ref{thm:zero}. To do this, we identify $V=\R^m\times\{0\}$
and perform an anisotropic blow-down, shrinking the horizontal components (in $\R^m$)
and leaving the vertical ones (in $\R^{n-m}$) unchanged. It is easy to see that such a blow-down disintegrates as
$dx'\otimes\kappa_{x'}$, where $x'$ denotes points in $\R^m$, while $(\kappa_{x'})_{x'\in\R^m}$
is a Young measure retaining information on the vertical oscillation of the initial $\mu$. Each $\kappa_{x'}$
is a probability measure on the sphere $\{q\in V^\perp:|q|=|b|\}$.

We then study the center-of-mass function $h(x'):=\int_{V^\perp}q\,d\kappa_{x'}(q)$ and its Gaussian averages $\PP_rh$.
It turns out that $h$ satisfies several identities, such as $$|b|^2-h\cdot\PP_rh=2r^2A_r\ge0$$ among others,
where $(A_r)_{r>0}$ is a family of auxiliary functions (given by $\lim_{L\to\infty}L^2(\PP_r\rho_L-1)$,
where $\rho_L:=S_{0,L}(\pi_V)_*\mu$, which is the projection of the anisotropic rescaling at scale $L$).

Since we can easily show that $r^2A_r\to0$ in $L^1_{loc}$ as $r\to0$, we have $|b|^2-h\cdot\PP_rh\to0$, which (together with the trivial bound $|h|\le |b|$)
forces $|h|=|b|$, so that a.e.\ $\kappa_{x'}$ must be a Dirac mass $\delta_{h(x')}$. The aforementioned identities
then give $h\in H^1_{loc}$, $\PP_rh(0)=0$, $\PP_r(|Dh|^2)(0)\le|b|^2/r^2$, and $\PP_r(h\cdot\PP_sh)(0)=|b|^2r^2/(r^2+s^2)$, which can be shown to imply $h\equiv0$, i.e., $|b|=0$, giving $\supp\mu\subseteq V$.

In Section \ref{sec:preiss}, we summarize the main arguments needed to deduce Theorems \ref{thm:preiss} and \ref{preiss.bis} from Theorem \ref{thm:zero}, for the sake of completeness; all the content of this section is actually well known, up to minor adaptations.

Finally, Section \ref{sec:quantitative} is devoted to Theorem \ref{thm:width}.
Given $\mu$ as in its statement, since we already know that the width $w_V(\mu)<\infty$
and the claimed bound is scale-invariant, we can assume $w_V(\mu)=1$ and reduce ourselves to showing that
$\mathcal{D}(\mu)$ cannot be arbitrarily small. It might be tempting to use a direct compactness argument,
but the issue is that the width does not pass to the limit.
On the other hand, we can perform the same anisotropic blow-down used before: in this way, we get a Young measure
$(\kappa_{x'})_{x'\in\R^m}$ where each $\kappa_{x'}$ is supported in an annular region $\{q\in V^\perp:|q|\in[1-C\sqrt{\mathcal{D}(\mu)},1]\}$. We also obtain approximate versions of the previous identities,
which then allow us to use a compactness argument reducing to the previous situation (in the limit $\mathcal{D}(\mu)\to0$),
for a suitable limiting $h$ with $|h|=1$, contradicting the conclusion that $h\equiv0$.

\subsection{Usage of AI tools} The statement and proof of Theorem \ref{thm:zero} were found by a commercial AI model,
once the author asked it to try to simplify the proof of the step ``flatness at infinity implies flatness''
in Preiss's proof using anisotropic blow-downs and Young measures. Besides rewriting the output, the author then obtained the quantitative version of it, namely Theorem \ref{thm:width}. The AI model was also used to generate the examples above and in the proofreading stage. The author takes full responsibility for all the contents of the paper.

\subsection{Acknowledgements}The author's research is funded by the European Research Council (ERC) through StG 101165368 ``MAGNETIC''.

\section{Preliminary facts for measures with finite discrepancy}\label{sec:prel}

Throughout the section, we will assume that $\mu$ is a positive Radon measure on $\R^n$ such that $G_R^\mu(x)<\infty$
for every $x\in\supp\mu$ and $R>0$, as well as $$\mathcal D(\mu)<\infty,$$ where we recall that $G_R^\mu$ and the Gaussian discrepancy $\mathcal{D}(\mu)$ were introduced in \eqref{eq:gaussian-measure}--\eqref{eq:defect-infinity}.
We first establish a few consequences of the assumption.
Up to a translation, without loss of generality we will also assume that
$0\in\Sigma:=\supp\mu$.

\begin{lemma}\label{lem:coarse}
There exist $R_0(\mu)\ge1$ and constants $c(n),C(n)>0$ such that
\begin{equation}\label{eq:coarse-growth}
 c r^m\le\mu(B_r(x))\le Cr^m
 \quad\text{for }x\in\Sigma,\ r\ge R_0,
\end{equation}
as well as
\begin{equation}\label{coarse.off}
\mu(B_r(x))\le Cr^m\quad\text{for }x\in\R^n,\ r\ge R_0.
\end{equation}
Further, every sequence $S_{x_j,R_j}\mu$, with $x_j\in\Sigma$ and
$R_j\to\infty$, has a subsequence converging to an
$m$-uniform measure, with local convergence of supports in the Hausdorff metric.
\end{lemma}

\begin{proof}
The assumption $\mathcal{D}(\mu)<\infty$ implies that,
for any $r>0$ large enough (uniformly in $x\in\Sigma$), we have $G_r^\mu(x)\in(1/2,2)$.
Since $G_r^\mu(x)\ge(4\pi)^{-m/2}e^{-1/4}\frac{\mu(B_r(x))}{r^m}$,
the upper bound in \eqref{eq:coarse-growth} follows.
Also, if $B_r(x)\cap\Sigma\neq\emptyset$ then we can take $x'$ in the intersection
and bound $\mu(B_r(x))\le\mu(B_{2r}(x'))\le C(n)(2r)^m$, giving \eqref{coarse.off} (up to enlarging $C$).

Given $a>1$ and $x\in\Sigma$, expanding $\R^n\setminus B_{ar}(x)=\bigcup_{k=1}^\infty[B_{(k+1)ar}(x)-B_{kar}(x)]$, we deduce that for $r$ large enough we have
\begin{align*}
\frac12\le G_r^\mu(x)&\le C(n,a)\frac{\mu(B_{ar}(x))}{r^m}+\sum_{k=1}^\infty e^{-k^2a^2/4}\frac{\mu(B_{(k+1)ar}(x))}{r^m}\\
&\le C(n,a)\frac{\mu(B_{ar}(x))}{r^m}+C(n)\sum_{k=1}^\infty e^{-k^2a^2/4}(k+1)^ma^m.
\end{align*}
Taking $a>1$ so large that the last series is at most $1/(4C(n))$, we obtain
$\mu(B_{ar}(x))\ge\frac{r^m}{4C(n,a)}$, which gives the lower bound in \eqref{eq:coarse-growth}.

Given rescalings $\mu_j:=S_{x_j,R_j}\mu$ as in the statement, the coarse bounds \eqref{eq:coarse-growth}
hold for all radii $r\ge{R_0}/{R_j}$, proving the existence of a subsequential limit
$\nu$ and convergence of supports. Since
$$G_r^{\mu_j}(x)=G_{R_jr}^\mu(x_j+R_jx)=1+O_{\mu,r}(R_j^{-2})$$
for any $x\in\supp\mu_j$ (i.e., $x_j+R_jx\in\supp\mu$), thanks to the assumption $\mathcal{D}(\mu)<\infty$,
in the limit we get $G_r^\nu(x)=1$ for any $x\in\supp\nu$ (by a truncation argument as above).
Rewriting
\begin{align*}G_r^\nu(x)&=\pi^{-m/2}(4r^2)^{-m/2-1}\int_0^\infty 2s e^{-s^2/(4r^2)}\nu(B_s(x))\,ds\\
&=\pi^{-m/2}(4r^2)^{-m/2-1}\int_0^\infty e^{-\sigma/(4r^2)}\nu(B_{\sqrt\sigma}(x))\,d\sigma,\end{align*}
we see that $F(\sigma):=\nu(B_{\sqrt\sigma}(x))$ has Laplace transform
$\int_0^\infty e^{-\tau\sigma}F(\sigma)\,d\sigma=\pi^{m/2}\tau^{-m/2-1}$, which is easily computed to agree with the transform of $\mathcal{H}^m(B_{\sqrt\sigma}^m(0))=\omega_m\sigma^{m/2}$. By uniqueness, we get
$\nu(B_s(x))=\omega_ms^m$ for all $x\in\supp\nu$, proving that $\nu$ is $m$-uniform.
\end{proof}

Given $R>0$, let us use the notation
\begin{equation}\label{eq:gaussian-moments}
 \E_R[\psi]:=(4\pi R^2)^{-m/2}\int_{\R^n}\psi(z)e^{-|z|^2/(4R^2)}\,d\mu(z),
 \quad Q_R:=\frac1{2R^2}\E_R[z\otimes z]
\end{equation}
(note that $\E_R[1]=G_R^\mu(0)=1+O_\mu(R^{-2})$ for $R\ge1$). The matrix $Q_R$ is normalized in such a way that it
equals the orthogonal projection $\pi_V$ when $\mu=\Hh^m\res V$.
Let us record a useful expansion.

\begin{lemma}\label{taylor}
For every fixed $x\in\R^n$ and $R\ge R_0$, we have
\begin{equation}\label{eq:quadratic-expansion}
 4R^2(G_R^\mu(x)-G_R^\mu(0))
 =2\E_R[z]\cdot x+x\cdot Q_Rx-|x|^2\E_R[1]+O_x(R^{-1}),
\end{equation}
where $O_x$ indicates an implicit constant depending on $x$ and $n$, but not on $\mu$.
\end{lemma}

\begin{proof}
Indeed, we have
$$G_R^\mu(x)=e^{-|x|^2/(4R^2)}\E_R [e^{z\cdot x/(2R^2)}]$$ and we can expand
$$e^{-|x|^2/(4R^2)}=1-\frac{|x|^2}{4R^2}+O_x(R^{-4}),\quad
\E_R [e^{z\cdot x/(2R^2)}]=\E_R[1]+\frac{\E_R[z\cdot x]}{2R^2}+\frac{\E_R[(z\cdot x)^2]}{8R^4}+\varepsilon_{x,R},$$
where $|\varepsilon_{x,R}|\le O_x(R^{-6})\E_R[|z|^3e^{|z||x|/(2R^2)}]\le O_x(R^{-6-m})\int_{\R^n}|z|^3e^{-|z|^2/(8R^2)}\,d\mu(z)=O_x(R^{-3})$
(the last inequality can be obtained by an annular decomposition as in the previous proof).
Note that, similarly, $\frac{\E_R[z\cdot x]}{2R^2}=O_x(R^{-1})$ and $\frac{\E_R[(z\cdot x)^2]}{8R^4}=O_x(R^{-2})$;
hence, multiplying these two expansions, we get \eqref{eq:quadratic-expansion}.
\end{proof}

We now show that, whenever one blow-down is a plane $V$, the blow-down is in fact unique
and $\Sigma,V$ are within a uniform neighborhood of each other.

\begin{lemma}\label{lem:tube}
Assume that $\mu_j:=S_{R_j}\mu\weak\Hh^m\res V$ (in duality with $C^0_c(\R^n)$) for some sequence $R_j\to\infty$.
Then $\pi_{V^\perp}\Sigma$ is bounded, we have the full convergence
\begin{equation}\label{eq:full-flat-limit}
 S_R\mu\weak\Hh^m\res V,\ Q_R\to \pi_V\quad\text{as }R\to\infty,
\end{equation}
and, for some $C=C(\mu)>0$,
\begin{equation}\label{eq:projection-net}
 \dist(p,\pi_V\Sigma)\le C\quad\text{for }p\in V.
\end{equation}
Further, if we also have $\mathcal{D}(\mu)=0$, then $b:=\lim_{R\to\infty}\E_R[z]$ exists and $b\in V^\perp$, and moreover
\begin{equation}\label{eq:zero-cylinder}
 |\pi_{V^\perp}x-b|^2=|b|^2\quad\text{for }x\in\Sigma.
\end{equation}
\end{lemma}

The last identity can be seen as an analogue of \cite[Theorem~3.6(1)]{Preiss} (in the special case $q=1$).

\begin{proof}
Using \eqref{coarse.off} for the rescalings $\mu_j$, it is straightforward to check that
$Q_{R_j}^\mu=Q_1^{\mu_j}\to Q_1^{\mathcal{H}^m\res V}=\pi_V$ (note that the analogue for $\E_{R_j}[z]$ does not hold, since
this term is not normalized). Evaluating
\eqref{eq:quadratic-expansion} at any $x\in\Sigma$, we then obtain a bound $|\E_{R_j}[z]\cdot x|\le C_x$
independent of $j$. Using finitely many such vectors $x$ forming a basis of $L:=\spanop\Sigma$,
we obtain a uniform bound for $\pi_L\E_{R_j}[z]$, which equals $\E_{R_j}[z]$ since trivially
$\E_{R_j}[z]\in L$.
We can then extract a subsequential limit
$b:=\lim_{j\to\infty}\E_{R_j}[z]\in L$.

For each fixed $x\in\Sigma$, writing $x=p+q$ with $p\in V$ and $q\in V^\perp$, using \eqref{eq:quadratic-expansion} we deduce
\begin{equation}\label{eq:approximate-cylinder}
 |2b\cdot x-|q|^2|=\lim_{j\to\infty}|2\E_{R_j}[z]\cdot x+x\cdot Q_{R_j}x-|x|^2\E_{R_j}[1]|\le8\mathcal{D}(\mu).
\end{equation}
In particular, $2b\cdot x\ge-8\mathcal{D}(\mu)$. For every $v\in V$, the fact that $\supp\mu_j\to V$
provides $x_j\in\Sigma$ with $x_j/R_j\to v$.
Since $b\cdot x_j/R_j\ge-4\mathcal{D}(\mu)/R_j$, in the limit we obtain $b\cdot v\ge0$.
Replacing $v$ with $-v$ shows that $b\perp V$, and therefore \eqref{eq:approximate-cylinder}
reads $||q-b|^2-|b|^2|=|2b\cdot q-|q|^2|\le8\mathcal{D}(\mu)$; in particular,
\begin{equation}\label{bd.vertical}
 |\pi_{V^\perp}x-b|^2\le |b|^2+8\mathcal{D}(\mu)\quad\text{for }x\in\Sigma,
\end{equation}
which shows that $\pi_{V^\perp}\Sigma$ is bounded.

Any other blow-down $\nu$, including those centered at varying support
points, is thus supported in $V$. By Lemma~\ref{lem:coarse} it is
also $m$-uniform, which easily implies that $\nu=\Hh^m\res V$
(indeed, by differentiation with respect to $\Hh^m\res V$,
we see that $\nu=\Hh^m\res\supp\nu$ and, if we had $\supp\nu\neq V$,
taking a relative ball $B\subset V\setminus\supp\nu$ with $\partial B$ touching $\supp\nu$ at a point $x_0$, we would have $\nu(B_r(x_0))\le\omega_m r^m/2+o(r^m)$, a contradiction). This proves \eqref{eq:full-flat-limit}.

The set $\pi_V\Sigma$ is closed, since $\Sigma$ is closed and the normal coordinates are bounded.
If \eqref{eq:projection-net} fails, choose $p_j\in V$ with
$d_j:=\dist(p_j,\pi_V\Sigma)\to\infty$, nearest points
$p'_j\in \pi_V\Sigma$, and $z'_j\in\Sigma$ with $\pi_Vz_j'=p'_j$. As established earlier, the measures
$S_{z'_j,d_j}\mu$ converge to $\Hh^m\res V$,
but their projected supports omit the open unit ball centered at the unit vector
$(p_j-p'_j)/d_j$, giving a contradiction in the limit.

If $\mathcal{D}(\mu)=0$, then \eqref{bd.vertical} is an equality and gives
\eqref{eq:zero-cylinder}. Also, evaluating \eqref{eq:quadratic-expansion} along any other sequence $R_j'\to\infty$, we get
$$0=\lim_{j\to\infty}4(R_j')^2(G_{R_j'}^\mu(x)-G_{R_j'}^\mu(0))
=\lim_{j\to\infty}2\E_{R_j'}[z]\cdot x-|\pi_{V^\perp}x|^2.$$
Thus, any subsequential limit of $\E_{R_j'}[z]$ has uniquely determined scalar products with all $x\in\Sigma$,
i.e., with all vectors in $L$; since it belongs to $L$, it is uniquely determined, showing that the full limit
$\lim_{R\to\infty}\E_R[z]$ exists.
\end{proof}

\section{Rigid case and proof of Theorem \ref{thm:zero}}\label{sec:function}

On $\R^m$, for any radius $r>0$, let $k_r:\R^m\to(0,\infty)$ denote the usual Gaussian kernel and
let $\PP_r$ denote the corresponding convolution operator:
\begin{equation}\label{eq:kernel}
 k_r(x):=(4\pi r^2)^{-m/2}e^{-|x|^2/(4r^2)},\quad
 \PP_rf:=k_r*f.
\end{equation}
This aligns with the convention, used throughout the paper, to use scale parameters (which have the dimension of length).
With this notation, $(\PP_{\sqrt{t}})_{t>0}$ is just the heat semigroup.
In particular, we have the identities
\begin{equation}\label{eq:scale-calculus}
 \PP_r\PP_sf=\PP_{\sqrt{r^2+s^2}}f,\quad
 \partial_r(\PP_rf)=2r\Delta (\PP_rf),\quad
 \int_{\R^m} |x|^2k_r(x)\,dx=2mr^2.
\end{equation}

\begin{lemma}\label{lem:function}
Let $f\in H^1_{\rm loc}(\R^m;\R^\ell)$ with $|f|=M$ a.e.\ for some constant $M\ge0$. Assume, for all $r>0$, that
\begin{gather}
 \PP_rf(0)=0,\quad
 \PP_r(|Df|^2)(0)\le\frac{M^2}{r^2},\label{eq:function-hypotheses}\\
 \lim_{s\to\infty} s\PP_r(f\cdot\PP_sf)(0)=0.
 \label{eq:function-correlation}
\end{gather}
Then $f\equiv0$.
\end{lemma}
\begin{proof}
Let $\Lambda_r:=rD(\PP_rf)(0)$. We first claim that
$$\lim_{r\to\infty}\Lambda_r=0.$$
Differentiating the convolution and using $|f|\le M$, we get
\begin{equation}\label{eq:first-linear-average}
 D(\PP_rf)(0)=\frac1{2r^2}\int_{\R^m} k_r(x)f(x)\otimes x\,dx,\quad \|D^2\PP_r f\|_{L^\infty}\le C(m)M/r^2.
\end{equation}
A simple Taylor expansion then gives $\PP_sf(x)=D(\PP_sf)(0)[x]+O(M|x|^2/s^2)$.
Inserting this into \eqref{eq:function-correlation}, we get
\[
 o_r(s^{-1})=\int_{\R^m}k_r(x)f(x)\cdot\PP_sf(x)\,dx=\frac{2r}{s}\ip{\Lambda_r}{\Lambda_s}+O(M^2r^2/s^2),
\]
where we used \eqref{eq:first-linear-average} to rewrite $D(\PP_sf)(0)\cdot\int_{\R^m}k_r(x)f(x)\otimes x\,dx$
as $\frac{2r}{s}\ip{\Lambda_r}{\Lambda_s}$.
For each fixed $r>0$, this implies $\ip{\Lambda_r}{\Lambda_s}\to0$ as $s\to\infty$.
Since the span of all matrices $\Lambda_r$ is finite-dimensional, this proves that indeed $\Lambda_s\to0$
as $s\to\infty$.

Setting $f_r(x):=f(rx)$, since $|f_r|\le M$ and
\begin{equation}\label{eq:rescaled-function}
	\PP_1(|Df_r|^2)(0)=r^2\PP_r(|Df|^2)(0)\le M^2,
\end{equation}
and moreover $k_1>0$ is bounded below by a positive constant on every compact set,
by Rellich compactness there exist radii $r_j\to\infty$ such that
$$f_{r_j}\to g\quad\text{in }L^2_{loc}(\R^m;\R^\ell)$$
and a.e., for some $g\in H^1_{loc}(\R^m;\R^\ell)$.
In particular, we have $|g|=M$ a.e. The convergence is also
strong in $L^2(k_1\,dx)$, since
\[
 \int_{|x|>R}|f_{r_j}-g|^2k_1(x)\,dx
 \le4M^2\int_{|x|>R}k_1(x)\,dx\to0
 \quad\text{as }R\to\infty
\]
uniformly in $j$. Since the assignments $v\mapsto \PP_1v(0)=\int_{\R^m} k_1v\,dx$ and
$v\mapsto D(\PP_1v)(0)=\frac12\int_{\R^m} k_1v\otimes x\,dx$
are continuous on this weighted space, we have
\begin{equation}\label{eq:unit-limit}
 |g|=M\text{ a.e.},\quad \PP_1g(0)=\lim_{j\to\infty}\PP_1f_{r_j}(0)=\lim_{j\to\infty}\PP_{r_j}f(0)=0,
\end{equation}
and also, recalling the definition of $\Lambda_r$, the previous bound \eqref{eq:rescaled-function}, and lower
semicontinuity of the Dirichlet energy,
\begin{equation}\label{eq:unit-limit2}
 D(\PP_1g)(0)=\lim_{j\to\infty} D(\PP_1f_{r_j})(0)
 =\lim_{j\to\infty}\Lambda_{r_j}=0,\quad
 \PP_1(|Dg|^2)(0)\le M^2.
\end{equation}

To conclude, we will show that $M=0$. For $\sigma\in(0,1)$, let
\[
 U_\sigma:=\PP_{\sqrt{1-\sigma}}(g),\quad
 J(\sigma):=\PP_{\sqrt\sigma}(|U_\sigma|^2)(0),\quad
 K(\sigma):=\PP_{\sqrt\sigma}(|DU_\sigma|^2)(0).
\]
Since
$\partial_\sigma \PP_{\sqrt\sigma}=\Delta \PP_{\sqrt\sigma}$ and
$\partial_\sigma U_\sigma=-\Delta U_\sigma$, the identity
$\Delta|v|^2=2v\cdot\Delta v+2|Dv|^2$ gives
\begin{equation}\label{eq:J-derivative}
J'(\sigma)=\PP_{\sqrt\sigma}(\Delta|U_\sigma|^2+\partial_\sigma|U_\sigma|^2)(0)
=\PP_{\sqrt\sigma}(2|DU_\sigma|^2)(0)=2K(\sigma),
\end{equation}
and similarly
\begin{equation}
  K'(\sigma)=2\PP_{\sqrt\sigma}(|D^2U_\sigma|^2)(0),\quad K''(\sigma)=4\PP_{\sqrt\sigma}(|D^3U_\sigma|^2)(0)\ge0.
\end{equation}
While these hold for $\sigma\in(0,1)$, at the endpoints we can use \eqref{eq:unit-limit}--\eqref{eq:unit-limit2} and the fact that $|U_\sigma|^2\to M^2$ in $L^1_{loc}(\R^m)$ as $\sigma\to 1$ (with $|U_\sigma|^2\le M^2$), obtaining
\[
 J(0^+)=K(0^+)=0,\quad J(1^-)=M^2.
\]
By Jensen's inequality, we can also bound
\[
 K(\sigma)\le\PP_{\sqrt\sigma}\PP_{\sqrt{1-\sigma}}(|Dg|^2)(0)
 =\PP_1(|Dg|^2)(0)\le M^2.
\]
Since $K$ is convex with $K(0^+)=0$ and $K(1^-)\le M^2$, we must have $K(\sigma)\le\sigma M^2$,
and therefore
\[
 M^2=\int_0^1 J'(\sigma)\,d\sigma=\int_0^1 2K(\sigma)\,d\sigma
 \le\int_0^1 2\sigma M^2\,d\sigma=M^2.
\]
The equality forces $K(\sigma)=\sigma M^2$, hence $K''(\sigma)=0$ and thus
$D^3U_\sigma=0$. We deduce that, for all $\sigma\in(0,1)$, $U_\sigma$ is a bounded polynomial of degree
at most two, hence constant. Consequently, by definition of $K$, we have $K(\sigma)=0$.
Since we obtained $K(\sigma)=\sigma M^2$ earlier, we arrive at $M=0$.
\end{proof}

We next record the form in which this lemma will be applied.

\begin{corollary}\label{lem:fiber}
Let $(\kappa_x)_{x\in\R^m}$ be a (weakly measurable) family of probability measures on the sphere
$\mathbb{S}:=\{v\in\R^\ell:|v|=M\}$ and let $h(x):=\int_{\mathbb{S}} v\,d\kappa_x(v)$.
Suppose that for each $r>0$ we have
\begin{equation} \label{eq:fiber-row}
2r^2A_r(x)=M^2-v\cdot \PP_rh(x)\quad\text{for }dx\otimes\kappa_x(v)\text{-a.e. }(x,v),
\end{equation}
for a suitable bounded smooth function $A_r:\R^m\to[0,\infty)$,
and assume moreover the identities
\begin{equation}\label{weird.ids}
 \PP_rA_s=A_{\sqrt{r^2+s^2}},\quad
 \PP_rh(0)=0,\quad
 A_r(0)=\frac{M^2}{2r^2}
\end{equation}
for every $r,s>0$. Then $\kappa_x=\delta_0$ for a.e.\ $x\in\R^m$.
\end{corollary}

\begin{proof}
Averaging \eqref{eq:fiber-row} in $v$ gives
\begin{equation}\label{eq:fiber-average}
 2r^2A_r=M^2-h\cdot \PP_rh.
\end{equation}
Note that the left-hand side goes to zero in $L^1_{loc}$ as $r\to0$, since by \eqref{weird.ids} we have
\[
 \int k_{\sqrt{4-s^2}}(x)A_s(x)\,dx=\PP_{\sqrt{4-s^2}}A_s(0)=A_2(0)
 \quad\text{for }s\in(0,2)
\]
and $k_{\sqrt{4-s^2}}$ is bounded below by a positive constant on any given compact set
when $s\in(0,1)$, which implies that $A_s$ is locally uniformly bounded in $L^1$ for $s\in(0,1)$.
Also, since $\PP_rh\to h$ in $L^2_{loc}$ (as $|h|\le M$), we have $h\cdot\PP_r h\to |h|^2$ in $L^1_{loc}$.
Hence, \eqref{eq:fiber-average} yields
$|h(x)|=M$ at a.e.\ $x$, at which points it follows that
\[
 \int_{\mathbb{S}} |v-h(x)|^2\,d\kappa_x(v)=M^2-|h(x)|^2=0
\]
and thus $\kappa_x=\delta_{h(x)}$ a.e. In particular, using also \eqref{eq:fiber-average}, for a.e.\ $x$ we get
\begin{equation}\label{eq:fiber-increment}
 \int_{\R^m} |h(x+z)-h(x)|^2k_r(z)\,dz=2M^2-2h(x)\cdot\int_{\R^m} h(x+z)k_r(z)\,dz=4r^2A_r(x).
\end{equation}

To obtain a gradient bound, we use the fact that the functions $z_i/(\sqrt2\,r)$ are
orthonormal in $L^2(k_r(z)\,dz)$, and note that
\[
 \int_{\R^m} \frac{z_i}{\sqrt2\,r}(h(x+z)-h(x))k_r(z)\,dz
 =\int_{\R^m} \frac{z_i}{\sqrt2\,r}h(x+z)k_r(z)\,dz
  =\sqrt2\,r\,\partial_{x_i}\PP_rh(x).
\]
Bessel's inequality (applied to each component of the target $\R^\ell$) therefore gives
\begin{equation}\label{eq:bessel}
 |D(\PP_rh)|^2\le\frac{1}{2r^2}\int_{\R^m}|h(x+z)-h(x)|^2k_r(z)\,dz=2A_r.
\end{equation}
Using \eqref{weird.ids} again, for all $r,s>0$ we deduce that
\[
 \PP_r(|D(\PP_sh)|^2)(0)
 \le \PP_r(2A_s)(0)
 =2A_{\sqrt{r^2+s^2}}(0)=\frac{M^2}{r^2+s^2}.
\]
Letting $s\to0$, since $\PP_s h\to h$ in $L^2_{loc}$ we deduce that
\begin{equation}\label{eq:fiber-energy}
 h\in H^1_{loc},\quad \PP_r(|Dh|^2)(0)\le M^2/r^2.
\end{equation}
Finally, from \eqref{eq:fiber-average} we get
\[
 \PP_r(h\cdot \PP_s h)(0)
 =\PP_r(M^2-2s^2A_s)(0)
 =M^2-2s^2A_{\sqrt{r^2+s^2}}(0)
 =M^2\frac{r^2}{r^2+s^2}.
\]
We can then apply Lemma~\ref{lem:function} with $f:=h$.
\end{proof}

\begin{proof}[Proof of Theorem~\ref{thm:zero}]
Without loss of generality, we can assume $0\in\Sigma:=\supp\mu$.
Since $\mu$ is fixed, in the sequel we will not indicate explicitly the dependence of error terms on $\mu$ and $n$.
Applying \eqref{eq:zero-cylinder} we find
$b\in V^\perp$ such that, setting $M:=|b|$, we have
\[
 |\pi_{V^\perp}x-b|^2=M^2\quad\text{for }x\in\Sigma.
\]
%
Let $\phi(x):=\pi_{V^\perp}x-b$. The previous identity says that $|\phi(x)|=M$ on all of $\Sigma$,
while Lemma \ref{lem:tube} gives $b=\lim_{R\to\infty}\E_R[z]$ and thus
\begin{equation}\label{avg.norm.to.zero}
\E_R[\phi(z)]\to0
\end{equation}
(recall that $\E_R[\cdot]$ was introduced in \eqref{eq:gaussian-moments}).

Let us identify $V$ with $\R^m=\R^m\times\{0\}\subset\R^n$.
For a scale $L\ge1$, we blow down the first $m$ components: let $\zeta_L:\R^n\to V\times V^\perp=\R^n$
be given by $\zeta_L(x):=(\pi_Vx/L,\phi(x))$ and
\begin{equation}\label{eq:anisotropic-measures}
 \eta_L:=L^{-m}(\zeta_L)_*\mu,
 \quad \rho_L:=L^{-m}(\pi_V/L)_*\mu=(\pi_V)_*\eta_L.
\end{equation}
By \eqref{eq:full-flat-limit}, we have $\rho_L\weak dx'$, where we use $x'$ to denote points in $\R^m$.
More quantitatively, the coarse upper bound \eqref{coarse.off} and the fact that $|\pi_{V^\perp}x-b|=M$ for $x\in\Sigma$ give
\begin{equation}\label{eq:marginal-growth}
 \rho_L(B_r(x'))\le C(n)(Lr+M)^m/L^m=C(n)(r+M/L)^m\quad\text{for }x'\in\R^m,\ r\ge R_0/L.
\end{equation}
After a subsequence, we can therefore extract a limit of $\eta_L$ and disintegrate it with respect to $\pi_V$:
\begin{equation}\label{eq:young-limit}
 \eta_L\weak dx'\otimes\kappa_{x'},
\end{equation}
where each $\kappa_{x'}$ is a probability measure on $\{u\in V^\perp:|u|=M\}$.
We will show that, in fact, $\kappa_{x'}=\delta_0$ for a.e.\ $x'$.

For $(x_0',u_0)\in\supp\eta_L$, the assumption $\mathcal{D}(\mu)=0$ at the original support
point $(Lx_0',u_0+b)\in\Sigma$ and scale $Lr$ reads
\begin{equation}\label{eq:zero-row}
 G_{L,r}(x_0',u_0):=\int_{\R^n} k_r(x'-x_0')e^{-|u-u_0|^2/(4L^2r^2)}\,d\eta_L(x',u)=1+o_r(L^{-2}).
\end{equation}
For $r$ in a compact subset of $(0,\infty)$, the error is uniform in
$r$ and in the support point. Letting $$g_{L,r}:\R^m\to(0,\infty),\quad g_{L,r}:=\PP_r\rho_L,$$ since on $\supp\eta_L$ it holds
$|u-u_0|\le2M$, on $L^{-1}\pi_V\Sigma$ it follows that
\begin{equation}\label{g.bds}
 1+o_r(L^{-2})\le g_{L,r}\le
 e^{M^2/(L^2r^2)}(1+o_r(L^{-2}));
\end{equation}
recall that, by \eqref{eq:projection-net}, $L^{-1}\pi_V\Sigma$ is a $(C(\mu)L^{-1})$-net in $\R^m$.
Also, \eqref{eq:marginal-growth} bounds all spatial
derivatives of $g_{L,r}$ uniformly in $L$, locally uniformly in $r>0$.
We claim that, as a consequence,
\begin{equation}\label{eq:net-interpolation}
 \norm{g_{L,r}-1}_{L^\infty}\le C(\mu,r)L^{-2}.
\end{equation}
Indeed, we can assume $m>0$, since when $m=0$ the claim follows from \eqref{g.bds}. Let $f:=g_{L,r}-1$ and $\delta:=C(\mu)L^{-1}$, so that $L^{-1}\pi_V\Sigma$ is a $\delta$-net. Fix $x'\in\mathbb R^m$ and choose a unit vector $e$ parallel to $\nabla f(x')$ (we can take any unit vector if $\nabla f(x')=0$). By the net property, there exist $p_\pm\in L^{-1}\pi_V\Sigma$ such that
$|p_\pm-(x'\pm \delta e)|\le \delta$.
Consequently,
$$|p_\pm-x'|\le2\delta,\quad
\nabla f(x')\cdot(p_+-x')\ge0,\quad
\nabla f(x')\cdot(p_--x')\le0.$$
Since by Taylor’s formula we have
$$f(p_\pm)=f(x')+\nabla f(x')\cdot(p_\pm-x')+\varepsilon_\pm,
\quad |\varepsilon_\pm|\le\frac{H|p_\pm-x'|^2}{2}\le2H\delta^2,$$
where $H:=\|D^2f\|_{L^\infty}$, using \eqref{g.bds} and the signs of the linear terms we obtain
$$
o_r(L^{-2})-2H\delta^2
\le f(p_-)-2H\delta^2
\le f(x')
\le f(p_+)+2H\delta^2
\le O_r(L^{-2})+2H\delta^2,
$$
which proves the claim \eqref{eq:net-interpolation} since the Hessian
of $f=g_{L,r}-1$ is bounded uniformly in $L,r$ (for $r$ in a compact set).

We now reconstruct the ingredients needed to apply Corollary \ref{lem:fiber}.
Let
$$h(x'):=\int_{V^\perp} u\,d\kappa_{x'}(u),$$
so that $|h|\le M$. We also define
\begin{equation}\label{eq:horizontal-defect}
 A_{L,r}:=L^2(g_{L,r}-1),\quad
 h_{L,r}(x_0'):=\int_{\R^n} k_r(x'-x_0')u\,d\eta_L(x',u).
\end{equation}
The identities $\PP_sA_{L,r}=A_{L,\sqrt{r^2+s^2}}$, $\partial_rA_{L,r}=2r\Delta A_{L,r}$,
and \eqref{eq:net-interpolation} give uniform bounds (in $L$) for any derivative of $A_{L,r}$ with respect to $x',r$, on compact subsets of $\{r>0\}$.
Hence, after a subsequence, we have $A_{L,r}\to A_r$ locally smoothly in
$(x',r)\in\R^m\times(0,\infty)$. The $L^\infty$ bound on $A_{L,r}$ for fixed $r>0$ allows us to interchange
the limit with Gaussian convolutions, giving
\begin{equation}\label{eq:limit-semigroup}
 \PP_sA_r=\lim_{L\to\infty}\PP_s(L^2(\PP_r\rho_L-1))=\lim_{L\to\infty}L^2(\PP_{\sqrt{r^2+s^2}}\rho_L-1)=A_{\sqrt{r^2+s^2}}.
\end{equation}
Likewise, we have $h_{L,r}\to \PP_rh$ locally smoothly, by definition of $h_{L,r}$ and
\eqref{eq:marginal-growth}.

Finally, expanding the term $e^{-|u-u_0|^2/(4L^2r^2)}=1-\frac{|u-u_0|^2}{4L^2r^2}+o_r(L^{-2})$ in \eqref{eq:zero-row}, we get
$$1+o_r(L^{-2})=G_{L,r}(x_0',u_0)=g_{L,r}(x_0')-\int_{\R^n}k_r(x'-x_0')\frac{|u-u_0|^2}{4L^2r^2}\,d\eta_L(x',u)
+o_r(L^{-2})$$
for any $(x_0',u_0)\in\supp\eta_L$, which rearranges to
\begin{equation}\label{eq:zero-expanded-row}
 2r^2A_{L,r}(x_0')=M^2g_{L,r}(x_0')-u_0\cdot h_{L,r}(x_0')+o_r(1),
\end{equation}
where $o_r(1)$ denotes an error which goes to zero as $L\to\infty$, uniformly on compact subsets of $\{r>0\}$.
As a consequence, for any fixed $r>0$, in the limit $L\to\infty$ we obtain
\begin{equation}\label{eq:zero-fiber-identity}
 2r^2A_r(x')=M^2-u\cdot\PP_rh(x')
 \quad\text{for }(x',u)\in\supp (dx'\otimes\kappa_{x'}).
\end{equation}
In particular, averaging in $u$ we get $2r^2A_r(x')=M^2-h(x')\cdot\PP_rh(x')$ for a.e.\ $x'\in\R^m$
and, since $|h|,|\PP_rh|\le M$, by continuity we have $A_r\ge0$ everywhere.

To conclude, we have
$$h_{L,r}(0)=\int_{\R^n} k_{Lr}(x')u\,d\mu(x',u+b)
=\E_{Lr}[e^{|\pi_{V^\perp}z|^2/(4L^2r^2)}\phi(z)];$$
since $\pi_{V^\perp}\Sigma$ is bounded, this equals $\E_{Lr}[\phi(z)]+O_r(L^{-2})$,
which tends to zero as $L\to\infty$ by \eqref{avg.norm.to.zero}. Therefore,
\begin{equation}\label{P.r.zero}
\PP_rh(0)=\lim_{L\to\infty}h_{L,r}(0)=0.
\end{equation}
The original support point $0\in\Sigma$ corresponds to
$(0,-b)\in\supp\eta_L$ for every $L>0$. Evaluating
\eqref{eq:zero-fiber-identity} at this point (note that \eqref{eq:zero-fiber-identity} holds also at limits of points in $\supp\eta_L$, as $L\to\infty$) and using the fact that $\PP_rh(0)=0$, we also get
\[
 A_r(0)=\frac{M^2}{2r^2}.
\]
Hence, we are under the assumptions of Corollary~\ref{lem:fiber}, which gives $M=0$, as required.
\end{proof}

The following immediate consequence was first obtained in \cite[Corollary~3.16]{Preiss}, as one of the key rigidity steps towards Theorem \ref{thm:preiss}.

\begin{corollary}\label{cor:uniform-flat}
An $m$-uniform measure $\mu$ with one flat blow-down is flat.
\end{corollary}
\begin{proof}
By $m$-uniformity, we compute that $G_R^\mu=1$ on $\supp\mu$. By Theorem~\ref{thm:zero}, $\supp\mu$ is included in an affine $m$-plane. Using $m$-uniformity again, we see that $\mu$ must agree with the Hausdorff measure on this plane
(see, e.g., the proof of Lemma \ref{lem:tube} for the standard argument).
\end{proof}

\section{Application to Preiss's theorem}\label{sec:preiss}

We now show how to apply Corollary \ref{cor:uniform-flat} in order to obtain Theorems \ref{thm:preiss}
and \ref{preiss.bis}. All the facts presented in this section are well known, but they are collected here
in order to make the paper more self-contained. We will follow closely the presentations of \cite{DeLellis,MattilaBook},
which provide streamlined and insightful expositions of these topics. We will also give references to earlier works where such tools originated.

Given $\alpha\ge0$, let $\U^\alpha$ denote the class of centered $\alpha$-uniform measures $\nu$, i.e., positive Radon measures on $\R^n$ such that $0\in\supp\nu$ and
\[
 \nu(B_r(y))=\omega_\alpha r^\alpha
 \quad\text{for }y\in\supp\nu,\ r>0,
\]
where $\omega_\alpha:=\pi^{\alpha/2}/\Gamma(1+\alpha/2)$, which is the volume of the $\alpha$-dimensional unit ball when $\alpha\in\N$.

Consider a measure $\mu$ as in the statement of Theorem \ref{thm:preiss}.
For $\mu$-a.e.\ $x$, using the notation $S_{x,r}\mu:=r^{-\alpha}(z\mapsto(z-x)/r)_*\mu$, the following standard facts hold (see, e.g., \cite[Proposition~3.4]{DeLellis} or \cite[Lemma~14.7]{MattilaBook} for the inclusion $\mathcal{T}_x\mu\subseteq\U^\alpha$ in (ii),
and \cite[Lemma~6.6]{DeLellis} or \cite[Theorem~14.16]{MattilaBook} for (iii), respectively):
\begin{itemize}
\item[(i)] $\theta(x):=\lim_{r\to0}\frac{\mu(B_r(x))}{\omega_\alpha r^\alpha}$ exists in $(0,\infty)$;
\item[(ii)] the family $S_{x,r}\mu/\theta(x)$ is sequentially precompact as
$r\to0$, namely for any sequence $r_j\to0$ there exists a subsequence (not relabeled) such that
$\lim_{j\to\infty}S_{x,r_j}\mu/\theta(x)$ exists in the weak-* topology,
and the set $\mathcal{T}_x\mu$ of all such limits satisfies
$$\mathcal{T}_x\mu\subseteq\U^\alpha;$$
\item[(iii)] if $\nu\in\mathcal{T}_x\mu$ then, for any $y\in\supp\nu$ and $s>0$, also $S_{y,s}\nu\in\mathcal{T}_x\mu$.
\end{itemize}

Elements in $\mathcal{T}_x\mu$ are called \emph{tangent measures} at $x$ (note that normalization conventions are slightly inconsistent throughout the literature).

\begin{remark}\label{metrization}
Note that, if we assume without loss of generality that $\mu(\R^n)<\infty$, then for $\mu$-a.e.\ $x$ (namely, for $x$ where (i) holds) there exists $C_x>0$ such that $\mu(B_R(x))\le C_x\theta(x)R^\alpha$ for all $R>0$.
In particular, the rescalings $S_{x,r}\mu/\theta(x)$ (and thus also the elements of $\mathcal{T}_x\mu$) belong to the set
\begin{equation}\label{metrizable}
\mathcal{M}_{C_x}^\alpha:=\{\nu\in\mathcal{M}(\R^n):\nu(B_R(0))\le C_xR^\alpha\text{ for all }R>0\},
\end{equation}
where $\mathcal{M}(\R^n)$ denotes the set of all positive Radon measures on $\R^n$. By standard functional analysis,
the weak-* topology of $\mathcal{M}(\R^n)$, restricted to $\mathcal{M}_{C_x}^\alpha$, can be metrized,
yielding a compact metric space.
\end{remark}

The following fact was first obtained in \cite[Theorem~1]{M64}.

\begin{proposition}\label{lem:flat-tangent}
The exponent $\alpha$ is an integer in $\{0,\dots,n\}$ and, at every point $x$ where (i)--(iii) above hold, $\mathcal T_x\mu$ contains a flat measure.
\end{proposition}

\begin{proof}Fix $x$ where (i)--(iii) hold.
Let $m$ be the minimal dimension of $\spanop(\supp\mu')$, as $\mu'\in\mathcal T_x\mu$ varies,
and let $\nu\in\mathcal T_x\mu$ be a minimizer, so that $m=\dim\spanop(\supp\nu)$. Let us identify $\spanop(\supp\nu)$ with $\R^m$.

We must have $\alpha\le m$, since otherwise we could cover any compact $K\subset\R^m$
with a family of balls $\{B_r(x_i)\}_{i\in I}$ with $|I|=O_K(r^{-m})$ and deduce that
$\nu(K)\le\omega_\alpha(2r)^\alpha\cdot|I|=O_K(r^{\alpha-m})$ (note that any ball $B_r(x)$ has $\nu(B_r(x))\le\omega_\alpha(2r)^\alpha$, since if it intersects $\supp\nu$ at $x'$ then $B_r(x)\subseteq B_{2r}(x')$), which would give $\nu(K)=0$ once we let $r\to0$, a contradiction.

Assume now that $\alpha<m$; we will reach a contradiction also in this case, proving that $\alpha=m$ is an integer.
We cannot have $\supp\nu=\R^m$, since otherwise we could place at least $c(m)r^{-m}$ disjoint balls of any given radius $r\in(0,1)$ inside $B_1(0)$, giving $\nu(B_1(0))\ge c(m)\omega_\alpha r^{\alpha-m}$ by uniformity,
a contradiction since $r^{\alpha-m}$ blows up as $r\to0$. We can then choose $p\not\in\supp\nu$
and a nearest point $p'\in\supp\nu$, so that the ball $\bar B_{|p'-p|}(p)$ touches $\supp\nu$ at $p'$.
Letting $v:=p'-p$, it follows that any tangent $\eta\in\mathcal T_{p'}\nu\subseteq\mathcal T_x\mu$
is supported in the half-space $\{z:\ip{z}{v}\ge0\}$.

By uniformity, the smooth function
\[
 H(y):=\int_{\R^m} e^{-|z-y|^2}\,d\eta(z)
\]
is constant on $\supp\eta$. Since $H(y)-H(0)=2\ip{w}{y}+O(|y|^2)$,
with $w:=\int_{\R^m} ze^{-|z|^2}\,d\eta(z)$,
if $w\neq0$ then any tangent at zero is supported in $w^\perp$ (as $H(y)-H(0)$ vanishes on $\supp\eta$).
On the other hand, if $w=0$ then
$$\int_{\R^m}\ip{z}{v}e^{-|z|^2}\,d\eta(z)=0,$$
giving $\ip{z}{v}=0$ for $z\in\supp\eta$ (since we had $\ip{z}{v}\ge0$), and hence $\supp\eta\subseteq v^\perp$.
In both cases, we contradict the minimality of $m$.

Thus, $\alpha=m$ is an integer and, by the same argument, $\supp\nu$ equals its span, an $m$-dimensional plane,
proving that $\nu$ is flat (see, e.g., the proof of Lemma~\ref{lem:tube}).
\end{proof}

The cases $m=0$ and $m=n$ in Theorem~\ref{thm:preiss} are very classical: when $m=0$, it is easy to check that
the measure is atomic, i.e., a finite or countable sum of (multiples of) Dirac masses, while
if $m=n$ then (by differentiation with respect to the Lebesgue measure $\LL^n$) we have $\mu\ll\LL^n$, matching the definition of rectifiable measure in these borderline cases. We
therefore assume $1\le m\le n-1$ for the rest of this section.

Given $\nu\in\U^m$, in the sequel we use again the notation $\E_R[\cdot]=\E_R^\nu[\cdot]$ from
\eqref{eq:gaussian-moments} and set
\[
 Q(\nu):=\frac12\E_1[z\otimes z]\ge0.
\]
Using uniformity, we compute that $\E_R[1]=1$ and
$\E_R[|z|^2]=2mR^2$, so that $\tr Q(\nu)=m$. Let us record how facts proved earlier simplify in this case, starting from the following special case of \cite[Theorem~3.6(1)]{Preiss}.

\begin{lemma}\label{lem:quadrics}
If $\nu\in\U^m$ and $\lambda$ is a blow-down, i.e., $S_{R_j}\nu\weak\lambda$ for a sequence of scales
$R_j\to\infty$, then there exists $b\in\spanop(\supp\nu)$ such that
\begin{gather}
 |x|^2=2b\cdot x+x\cdot Q(\lambda)x
       \quad\text{for }x\in\supp\nu,\label{eq:original-quadric}\\
 |y|^2=y\cdot Q(\lambda)y
       \quad\text{for }y\in\supp\lambda.\label{eq:limiting-quadric}
\end{gather}
\end{lemma}
\begin{proof}
Since $G_R^\nu\equiv1$ on the support, \eqref{eq:quadratic-expansion}
becomes
\[
 2\E_R[z]\cdot x
 +x\cdot Q(S_R\nu)x=|x|^2+O_x(R^{-1})
 \quad\text{for }x\in\supp\nu.
\]
By uniformity of $S_{R_j}\nu$ and fast decay of the Gaussian weight,
we have $Q(S_{R_j}\nu)\to Q(\lambda)$. As in the proof of Lemma \ref{lem:tube}, we get
$\E_{R_j}[z]\to b$, so that
\eqref{eq:original-quadric} holds at every $x\in\supp\nu$.
For $y\in\supp\lambda$, choose $x_j\in\supp\nu$ such that
$x_j/R_j\to y$. Using these points in \eqref{eq:original-quadric} and dividing by $R_j^2$, we deduce
\eqref{eq:limiting-quadric}.
\end{proof}

On the set of positive Radon measures $\mathcal{M}(\R^n)$, consider the functional $\mathcal F:\mathcal{M}(\R^n)\to[0,\infty]$ given by
\begin{equation}\label{eq:flatness-functional}
 \mathcal F(\eta):=\min_{V\in\mathrm{Gr}_m(\R^n)}
       (4\pi)^{-m/2}\int_{\R^n} e^{-|z|^2/4}\dist(z,V)^2\,d\eta(z),
\end{equation}
where $\mathrm{Gr}_m(\R^n)$ is the set of all $m$-planes passing through the origin.
It is easy to check that $\mathcal F|_{\U^m}$ is finite, continuous, and vanishes precisely at the flat measures.
Let $\mathcal Q$ be the class of $\lambda\in\U^m$ satisfying
\eqref{eq:limiting-quadric} with their own matrix $Q(\lambda)$.
By Lemma~\ref{lem:quadrics}, every blow-down of a uniform measure
belongs to $\mathcal Q$. The following is essentially the fundamental rigidity result \cite[Theorem 3.14(1)]{Preiss}
observed by Preiss for blow-downs.

\begin{lemma}\label{lem:gap}
There exists $\delta_0(n)>0$ such that any
$\lambda\in\mathcal Q$ with $\mathcal F(\lambda)\le\delta_0$
must be flat.
\end{lemma}

\begin{proof}
Assuming that the statement is false, there exist non-flat measures $\lambda_j\in\mathcal Q$ with
$\mathcal F(\lambda_j)\to0$. After a subsequence, they converge to a limit $\bar\lambda$ with $\mathcal F(\bar\lambda)=0$, which is then flat. In other words, there exists an $m$-plane $V$ such that
$$\lambda_j\weak\Hh^m\res V,\quad Q(\lambda_j)\to \pi_V.$$
Let us list the eigenvalues of $Q(\lambda_j)$ in decreasing order, as
$\alpha_{j,1}\ge\cdots\ge\alpha_{j,n}\ge0$, and choose an
orthonormal eigenbasis $\{e_{j,1},\ldots,e_{j,n}\}$ of $Q(\lambda_j)$.

If $\alpha_{j,m}\ge1$, then the fact that $\sum_k\alpha_{j,k}=\tr Q(\lambda_j)=m$ (with $\alpha_{j,k}\ge0$) forces
$\alpha_{j,1}=\dots=\alpha_{j,m}=1$ and $\alpha_{j,m+1}=\dots=\alpha_{j,n}=0$.
Thus, in this case $Q(\lambda_j)$ is the orthogonal projection to an $m$-plane $V_j$ and, by
\eqref{eq:limiting-quadric}, we have $\supp\lambda_j\subseteq V_j$, contradicting our assumption that $\lambda_j$ is not flat. We can then assume $\alpha_{j,m}=1-\delta_j$ for some $\delta_j>0$.

Using again the fact that $\tr Q(\lambda_j)=m$, we get
\[
 \alpha_{j,k}\le m-(m-1)(1-\delta_j)=1+(m-1)\delta_j\quad\text{for }k<m
\]
(since all other summands in $\sum_{k=1}^m\alpha_{j,k}\le m$ are at least $1-\delta_j$), as well as
$\alpha_{j,k}<1$ for $k>m$. Thus, given $x\in\supp\lambda_j$,
writing $x_k:=\langle x, e_{j,k}\rangle$ for simplicity, using the assumption $\lambda_j\in\mathcal{Q}$ we find
\[
 0=\sum_{k=1}^n(\alpha_{j,k}-1)x_k^2
 \le\delta_j\biggl((m-1)\sum_{k<m}x_k^2-x_m^2\biggr).
\]
We deduce that $\supp\lambda_j\subseteq\{x\in\R^n:x_m^2\le(m-1)\sum_{k<m}x_k^2\}$,
a set which excludes $B_\rho(e_{j,m})$ for some small enough $\rho(m)>0$. However, this contradicts
the fact that $\supp\lambda_j$ converges to an $m$-plane $V$ containing $\lim_{j\to\infty}e_{j,m}$ (after a subsequence).
\end{proof}

We are now ready to deduce Theorem \ref{thm:preiss} from Corollary \ref{cor:uniform-flat}.

\begin{proof}[Proof of Theorem~\ref{thm:preiss}]
By restricting $\mu$ to a ball, we can assume without loss of generality that $\mu(\R^n)<\infty$.
Fix a point $x\in\supp\mu$ where the properties (i)--(iii) listed above hold.
Besides integrality of the dimension, Proposition~\ref{lem:flat-tangent} provides a flat measure $\zeta\in\mathcal T_x\mu$.
Now assume by contradiction that there exists a non-flat $\xi\in\mathcal{T}_x\mu$. Thus, letting
\[
 \mu_r:=S_{x,r}\mu/\theta(x),
\]
there exist sequences $r_j\to0$ and $s_j\to0$, with $r_{j+1}<s_j<r_j$, such that
$$\mu_{r_j}\weak\zeta,\quad \mu_{s_j}\weak\xi.$$
The proof of the theorem is based on a continuity argument, exploiting the fact that $\mathcal{F}(\mu_r)$
is a continuous function of $r>0$ (note that $\mathcal{F}$ is finite and continuous on the class $\mathcal{M}_{C_x}^m$ from \eqref{metrizable}).

Combining Lemmas~\ref{lem:quadrics} and \ref{lem:gap} with
Corollary~\ref{cor:uniform-flat} proves that, whenever $\nu\in\U^m$ and $\liminf_{R\to\infty}\mathcal F(S_R\nu)<\delta_0$, in fact $\nu$ must be flat.
Fixing any $0<\delta<\min\{\delta_0,\mathcal F(\xi)\}$, since $\mathcal{F}(\mu_{s_j})\to\mathcal{F}(\xi)$
we must have $\mathcal{F}(\mu_{s_j})>\delta$ eventually, and similarly $\mathcal{F}(\mu_{r_j})<\delta$.
Thus, for $j$ large enough we can let $\sigma_j:=\max\{t\in[s_j,r_j]:\mathcal{F}(\mu_t)\ge\delta\}$.
After a subsequence, $\mu_{\sigma_j}\weak\upsilon\in\U^m$ with
$\mathcal F(\upsilon)=\delta$. We claim that $\mathcal F(S_R\upsilon)\le\delta$ for any $R\ge1$ as well,
which implies that $\upsilon$ is flat, and thus a contradiction.

To check this, note that $r_j/\sigma_j\to\infty$: indeed, if we had $r_j/\sigma_j\to \gamma\in[1,\infty)$
along a subsequence, then we would have $\zeta=S_\gamma\upsilon$ (as $\mu_{r_j}=S_{r_j/\sigma_j}\mu_{\sigma_j}$), which is impossible since $\zeta$ is flat while $\upsilon$ is not. As a consequence, we indeed have
$$\mathcal{F}(S_R\upsilon)=\lim_{j\to\infty}\mathcal{F}(S_R\mu_{\sigma_j})=\lim_{j\to\infty}\mathcal{F}(\mu_{R\sigma_j})\le\delta$$
for any $R\ge1$, since eventually $\sigma_j\le R\sigma_j<r_j$ and, by definition of $\sigma_j$,
$\mathcal{F}(\mu_t)\le\delta$ for all $t\in[\sigma_j,r_j]$.

We thus proved that all tangents at $x$ are flat. By the rectifiability criterion of Marstrand--Mattila from \cite{Marstr,Matt} (see also \cite[Chapter~5]{DeLellis} or \cite[Theorem~16.7]{MattilaBook}, or the argument at the end of the next subsection
together with the appendix), we conclude that $\mu$ is rectifiable.
\end{proof}

\subsection{Proof of Theorem \ref{preiss.bis}}
In this last part of the section, we just mention how the previous arguments need to be modified
in order to obtain Theorem \ref{preiss.bis}. First of all, we introduce a class of \emph{near-uniform} measures:
given any $\varepsilon>0$, we set
$$\U_\varepsilon^\alpha:=\{\nu\in\mathcal{M}(\R^n):0\in\supp\nu,\ \omega_\alpha r^\alpha\le\nu(B_r(x))\le(1+\varepsilon)\omega_\alpha r^\alpha\text{ for }x\in\supp\nu,\ r>0\}.$$
Let $\mathcal{G}$ be the set of good points where the analogues of properties (i)--(iii) above hold, namely:
\begin{itemize}
\item[(i')] $0<\theta(x):=\Theta_*^\alpha(\mu,x)\le\Theta^{\alpha,*}(\mu,x)<\infty$
with $\Theta^{\alpha,*}(\mu,x)\le(1+\epsilon_0)\Theta_*^{\alpha}(\mu,x)$;
\item[(ii')] the family $S_{x,r}\mu/\theta(x)$ is sequentially precompact as
$r\to0$ and the set $\mathcal{T}_x\mu$ of all possible limits $\lim_{j\to\infty}S_{x,r_j}\mu/\theta(x)$ satisfies
$$\mathcal{T}_x\mu\subseteq\U^\alpha_{\eps_0};$$
\item[(iii')] if $\nu\in\mathcal{T}_x\mu$ then, for any $y\in\supp\nu$ and $s>0$, also $S_{y,s}\nu\in\mathcal{T}_x\mu$ (same as (iii)).
\end{itemize}
Here, $\eps_0$ is the same constant appearing in \eqref{eq:preiss.bis}.
Again we have $\mu(\R^n\setminus\mathcal{G})=0$ and, without loss of generality, we can assume $\mu(\R^n)<\infty$.

The fact that $\alpha\in\N$ for uniform measures, together with a simple compactness argument,
yields $\varepsilon(\alpha,n)>0$ small enough such that, whenever $\nu\in\U_\varepsilon^\alpha$,
we must have $\alpha\in\N$. Thus, up to taking $\eps_0(\alpha,n)>0$ small enough,
we can assume that $\alpha=m\in\N$, and actually $m\in\{1,\dots,n-1\}$. The following is an analogue of \cite[Corollary~4.9]{Preiss}.

\begin{lemma}
Given any $\delta\in(0,\delta_0)$ (where $\delta_0(n)$ is given by Lemma \ref{lem:gap}),
up to taking $\eps_0$ small enough, we have $\mathcal{F}(\nu)\le\delta$ for all $x\in\mathcal{G}$ and $\nu\in\mathcal{T}_x\mu$.
\end{lemma}

\begin{proof}
Recall that, given $\delta\in(0,\delta_0)$, for an $m$-uniform measure $\nu\in\U^m$ we cannot have
$$\mathcal{F}(\nu)=\delta,\quad\mathcal{F}(S_R\nu)\le\delta\text{ for }R\ge1.$$
Another compactness argument yields $L(\delta,n)>1$ and $\eps(\delta,n)>0$ such that,
whenever $\nu\in\U_\eps^m$ and $\mathcal{F}(\nu)=\delta$, then
\begin{equation}\label{delta}
\mathcal{F}(S_R\nu)>\delta\quad\text{for some }R\in[1,L];
\end{equation}
for this $L$, by compactness again, there exists also $\tau(\delta,n)>0$ such that
\begin{equation}\label{tau}
\mathcal{F}(S_R\nu)>\tau\quad\text{for all }R\in[1,L],
\end{equation}
up to shrinking $\eps$ further.

Next, there exists a tangent measure $\zeta\in\mathcal{T}_x\mu$ such that $\mathcal{F}(\zeta)<\tau$ (again, up to shrinking $\eps$). Indeed, if we had a sequence of measures $\mu_j$ satisfying \eqref{eq:preiss.bis}
with parameters $\epsilon_j\to0$, but violating this conclusion at generic points $x_j$,
we could consider the set $\mathcal{T}$ given by the Kuratowski upper limit of the sets $\mathcal{T}_{x_j}\mu_j$ (i.e., the set of all possible subsequential limits of sequences $\nu_j\in\mathcal{T}_{x_j}\mu_j$).
Note that $\mathcal{T}\subseteq\U^m$ is compact and, whenever $\nu\in\mathcal{T}$, $y\in\supp\nu$, and $r>0$, we have $S_{y,r}\nu\in\mathcal{T}$. Hence, we could repeat the proof of Proposition \ref{lem:flat-tangent} and find a flat $\nu\in\mathcal{T}$. This would mean that, along a subsequence,
there exist $\nu_j\in\mathcal{T}_{x_j}\mu_j$ such that $\nu_j\weak\nu$, and in particular
$\mathcal{F}(\nu_j)\to\mathcal{F}(\nu)=0$, a contradiction, proving the claim.

Finally, we claim that, as a consequence of the previous observations, all tangent measures $\nu\in\mathcal{T}_x\mu$ have $\mathcal{F}(\nu)\le\delta$.
If this failed then, using the notation from the previous proof, we could find sequences $r_j,s_j\to0$
with $r_{j+1}<s_j<r_j$, $\mathcal{F}(\mu_{r_j})<\tau$, and $\mathcal{F}(\mu_{s_j})>\delta$,
since there would exist two tangent measures $\zeta,\xi$ where $\mathcal{F}(\zeta)<\tau$ and $\mathcal{F}(\xi)>\delta$, respectively. Taking $\sigma_j\in(s_j,r_j)$ as in the previous proof,
for $j$ large enough we should have $\mathcal{F}(\mu_{R\sigma_j})>\tau$ for all $R\in[1,L]$,
since otherwise we would contradict \eqref{tau} in the limit $j\to\infty$; in particular, $r_j\ge L\sigma_j$.

Thus, by \eqref{delta}, for $j$ large enough there would exist $R\in[1,L]$ (depending on $j$) such that
$\mathcal{F}(\mu_{R\sigma_j})>\delta$, contradicting the definition of $\sigma_j$ since $R\sigma_j\in(\sigma_j,r_j)$.
\end{proof}

Consider now any compact, purely unrectifiable set $F\subseteq\mathcal{G}$. It suffices to show that $\mu(F)=0$, which we will do by invoking a known rectifiability criterion, namely Proposition \ref{thm:Cprime}, proved in the appendix for the sake of completeness.

By standard measure theory, up to shrinking $F$ to another compact set (with positive measure if $\mu(F)>0$), we can assume without loss of generality that $\theta(x)\in[t,3t/2]$ as $x\in F$ varies,
for some uniform $t>0$, so that $\Theta_*^m(\mu,x)\ge t$ and $\Theta^{m,*}(\mu,x)\le 2t$,
and actually we can further assume that
\begin{equation}\label{growth.bds}
t/2\le\frac{\mu(B_r(x))}{\omega_m r^m}\le 4t\quad\text{for }x\in F,\ r\in(0,r_0),
\end{equation}
as well as $\mu(B_r(x))\le C_0r^m$ for all $x\in F$ and $r>0$ (for uniform $r_0,C_0>0$).
In particular, $\mu_{x,r}:=S_{x,r}\mu/\theta(x)$ belongs to the class $\mathcal{M}_{C_0/t}^m$ introduced in Remark \ref{metrization}, so that we can fix a distance $\mathfrak{d}$ metrizing the weak-* topology here. Up to shrinking $F$ (and $r_0$) further, we can also assume that convergence to tangent measures is uniform in $x$, namely that
\begin{equation}\label{delta.close}
\mathfrak{d}(\mu_{x,r},\mathcal{T}_x\mu)\le\delta\quad\text{for }x\in F,\ r\in(0,r_0).
\end{equation}
In the rest of the proof, we will tacitly
assume that $\delta$ (and thus $\eps_0$ in \eqref{eq:preiss.bis} and $r_0$ in \eqref{delta.close}) is small enough, so that the corresponding conclusions of each subsequent step hold true.

Fix $\eta<c(m)/8$ in the statement of Proposition \ref{thm:Cprime}. Note that its assumption \eqref{eq:growth} follows immediately from \eqref{growth.bds}.
Moreover, since $\mathcal{F}(\nu)\le\delta$ for all $\nu\in\mathcal{T}_x\mu$,
using \eqref{delta.close} and a compactness argument (in which $\delta\to0$) we see that $\mu_{x,r}$ is close to a suitable flat measure $\HH^m\res V$.
In particular, \eqref{eq:tube} holds for any $G\subseteq F$, since any point $y$
in $F\cap B_{2r}(x)$ outside of $B_{\eta r}(V+x)$ would have $\mu(B_{\eta r/2}(y))\ge t\cdot\frac{\omega_m\eta^m}{2^{m+1}}r^m$,
which is impossible in the limit $\delta\to0$ (since $(\HH^m\res V)(B_{\eta/2}(z))=0$ whenever $z\not\in B_\eta(V)$).

By standard differentiation of measures, we can now find a compact $G\subseteq F$, with $\mu(G)\ge\mu(F)/2$, such that $$\mu(B_{2r}(x)\setminus F)<t\cdot\omega_m\eta^m r^m/2\quad\text{for }x\in G,\ r\in(0,r_1),$$
for a suitable $r_1\in(0,r_0)$. It follows that \eqref{eq:filling} holds, up to replacing $r_0$ with $r_1$: indeed, given $x\in G$ and $r\in(0,r_1)$, for any $v\in V\cap B_1(0)$ we have $\mu(B_{\eta r}(x+rv))>t\cdot\omega_m\eta^m r^m/2$ (since in the limit $(\HH^m\res V)(B_{\eta}(v))=\omega_m\eta^m$) and $B_{\eta r}(x+rv)\subseteq B_{2r}(x)$. By our choice of $G$, we must have $B_{\eta r}(x+rv)\cap F\neq\emptyset$, i.e., $x+rv\in B_{\eta r}(F)$, as claimed.
Proposition \ref{thm:Cprime} then gives $\mu(G)=0$, and hence $\mu(F)=0$, as desired.

\section{The quantitative width estimate}\label{sec:quantitative}
We now prove Theorem~\ref{thm:width} for finite
$\mathcal D(\mu)$. As usual, we let $\Sigma:=\supp\mu$ and assume without loss of generality that $\mu\neq0$.
Recalling the assumption that a blow-down is flat, namely $\Hh^m\res V$, Lemma~\ref{lem:tube} then shows that
$\pi_{V^\perp}\Sigma$ is bounded. Letting
$$K:=\overline{\pi_{V^\perp}\Sigma},$$
let $w\ge0$ be the smallest radius of a ball $\bar B_w(q)\supseteq K$ in $V^\perp$.
Up to a translation in $V^\perp$ and in $V$, we can assume that $q=0$ and $0\in \pi_V\Sigma$;
let $\hat q_0\in V^\perp$ be such that $\hat x_0:=(0,\hat q_0)\in\supp\mu$, where we identify $\R^n=V\times V^\perp$.

Since the claimed bound trivially holds for $w=0$ and is easily seen to be scale-invariant, we can assume from now on that $w=1$.
Given $R>0$, let $\beta_R$ be the probability measure on $K$ given by
\[
 \int_K\psi\,d\beta_R:=
 \frac{(4\pi R^2)^{-m/2}}{G_R^\mu(\hat x_0)}
 \int_{\R^n} \psi(q)e^{-(|p|^2+|q-\hat q_0|^2)/(4R^2)}\,d\mu(p,q),
\]
where we write points $x\in\R^n$ as $(p,q)\in V\times V^\perp$.
Let $\beta$ be any subsequential limit of $\beta_R$ as $R\to\infty$.

As shown in the proof of Lemma \ref{lem:tube},
there exists a vector $b\in V^\perp$, actually given by the center of mass of $\beta$,
such that a variant of \eqref{eq:approximate-cylinder} holds: namely, translating $\mu$ by $-\hat x_0$,
we obtain a measure $\nu$ with $0\in\supp\nu$, so that \eqref{eq:quadratic-expansion} reads
$$4R_j^2(G_{R_j}^\nu(y)-G_{R_j}^\nu(0))=2\E_{R_j}^\nu[z]\cdot y+y\cdot Q_{R_j}^\nu y-|y|^2\E_{R_j}^\nu[1]+o_y(1).$$
Noting that $\E_{R_j}^\nu[z]\to b-\hat q_0$, $Q_{R_j}^\nu\to\pi_V$, and $\E_{R_j}^\nu[1]\to1$, this is the same as
$$4R_j^2(G_{R_j}^\mu(x)-G_{R_j}^\mu(\hat x_0))
=2(b-\hat q_0)\cdot(x-\hat x_0)+(x-\hat x_0)\cdot\pi_V(x-\hat x_0)-|x-\hat x_0|^2+o_x(1),$$
and the right-hand side is simply $2(b-\hat q_0)\cdot(q-\hat q_0)-|q-\hat q_0|^2+o_x(1)$.
Taking the difference at two points $x',x''\in\Sigma$ and passing to the limit, we get
\begin{equation}\label{eq:quant-oscillation}
 \operatorname{osc}_{q\in K}(|q|^2-2b\cdot q)\le8\mathcal{D}(\mu).
\end{equation}

Recalling that $\bar B_1(0)$ is the smallest ball including $K$,
we see that
$$0\in\operatorname{co}(K\cap\partial B_1(0));$$
indeed, if this were not the case, we could find a unit vector $v\in V^\perp$
such that $\ip{q}{v}>0$ for all $q\in K\cap\partial B_1(0)$, which would easily imply
that $K$ is included in a ball of smaller radius, centered at $\eps v$, for $\eps>0$ small enough:
otherwise, there would exist $q_\eps\in K$ with $|q_\eps-\eps v|^2\ge1$, i.e., $-2\eps\langle q_\eps,v\rangle+\eps^2
\ge1-|q_\eps|^2\ge0$, giving a limiting $q_0\in K$ with $|q_0|=1$ and $\langle q_0,v\rangle\le0$, impossible.

Consequently, there exist finitely many $q_i\in K\cap\partial B_1(0)$ and $\alpha_i\ge0$ such that
$\sum_i\alpha_i=1$ and $\sum_i\alpha_i q_i=0$. This gives
$\sum_i\alpha_i(|q_i|^2-2b\cdot q_i)=1$, which, combined with \eqref{eq:quant-oscillation}, implies
\begin{equation}\label{q.bd}
||q|^2-2b\cdot q-1|\le\sum_i\alpha_i\cdot|(|q|^2-2b\cdot q)-(|q_i|^2-2b\cdot q_i)|\le8\mathcal{D}(\mu)\quad\text{for }q\in K.
\end{equation}
Integrating against $\beta$, we deduce that
\[
 1-8\mathcal{D}(\mu)\le\int_K (|q|^2-2b\cdot q)\,d\beta=\int_K |q|^2\,d\beta-2|b|^2
 \le 1-2|b|^2.
\]
Since the subsequential limit $\beta$ was arbitrary, we have shown that
\begin{equation}
 \limsup_{R\to\infty}\biggl\lvert\int_K q\,d\beta_R\biggr\rvert
 \le2\sqrt{\mathcal{D}(\mu)}.\label{eq:quant-barycenter}
\end{equation}
Recalling \eqref{q.bd}, we then get
\begin{equation}
 1-4\sqrt{\mathcal{D}(\mu)}-8\mathcal{D}(\mu)\le |q|^2\le1\quad\text{for }q\in K.\label{eq:quant-shell}
\end{equation}

We now repeat the horizontal rescaling from the proof of Theorem \ref{thm:zero}:
as in \eqref{eq:anisotropic-measures}, identifying $V$ with $\R^m=\R^m\times\{0\}\subset\R^n$,
we now let $\zeta_L(x):=(\pi_Vx/L,\pi_{V^\perp}x)$ (in other words, we no longer translate by $-b$ vertically)
and
\[
 \eta_L:=L^{-m}(\zeta_L)_*\mu,
  \quad \rho_L:=L^{-m}(\pi_V/L)_*\mu=(\pi_V)_*\eta_L.
\]
Letting $L\to\infty$, after a subsequence, by \eqref{eq:full-flat-limit} we have
\begin{equation}\label{eq:quant-young}
 \rho_L\weak dx',\quad \eta_L\weak dx'\otimes\kappa_{x'},
\end{equation}
where $\kappa_{x'}$ is a probability measure supported in $K\subseteq\bar B_1(0)$.
We also set
$$h(x'):=\int_{K} q\,d\kappa_{x'}(q),\quad
 \sigma(x'):=\int_K |q|^2\,d\kappa_{x'}(q).$$
By \eqref{eq:quant-shell}, we also have
\begin{equation}\label{eq:quant-moment-bounds}
 |h|\le 1,\quad 1-4\sqrt{\mathcal{D}(\mu)}-8\mathcal{D}(\mu)\le \sigma\le 1.
\end{equation}
As before, we define
\[
 g_{L,r}:=\PP_r\rho_L,\quad A_{L,r}:=L^2(g_{L,r}-1),\quad
 h_{L,r}(x_0')=\int_{\R^n} k_r(x'-x_0')q\,d\eta_L(x',q).
\]
For $(x_0',q_0)\in\supp\eta_L$, setting again
\[
 G_{L,r}(x_0',q_0):=\int_{\R^n} k_r(x'-x_0') e^{-|q-q_0|^2/(4L^2r^2)}\,d\eta_L(x',q)
 =G_{Lr}^\mu(Lx_0',q_0),
\]
we have $|L^2r^2(G_{L,r}-1)|\le \mathcal{D}(\mu)+o_r(1)$ by definition of $\mathcal{D}(\mu)$ (with an error which is uniform over the support center and compact subsets of $\{r>0\}$). Since $|q-q_0|\le2$ on $\supp\eta_L$,
while we can assume $\mathcal{D}(\mu)\le1$ without loss of generality, we see that \eqref{eq:net-interpolation}
is still valid, with the same proof. In turn, this allows us to extract limits $A_r=\lim_{L\to\infty}A_{L,r}$
(after a subsequence), locally smoothly in $(x',r)$, with
\begin{equation}\label{a.lb}A_r\ge -\mathcal{D}(\mu)/r^2,\quad \PP_sA_r=A_{\sqrt{s^2+r^2}}.\end{equation}
Again, we also have $h_{L,r}\to \PP_rh$.

The same computation used to get \eqref{eq:zero-expanded-row} now gives
\begin{align*}
 4r^2A_{L,r}(x_0')&=\int_{\R^n} k_r(x'-x_0')|q|^2\,d\eta_L(x',q)
 -2q_0\cdot h_{L,r}(x_0')+|q_0|^2g_{L,r}(x_0')\\&\quad+4L^2r^2(G_{L,r}(x_0',q_0)-1)+o_r(1).
\end{align*}
Passing to the limit as $L\to\infty$, we obtain
\begin{equation}\label{eq:quant-row}
 |4r^2A_r(x')-\PP_r \sigma(x')+2q\cdot \PP_rh(x')-|q|^2|\le4\mathcal{D}(\mu)
\end{equation}
for $(x',q)$ in the support of $dx'\otimes\kappa_{x'}$. By \eqref{eq:quant-shell} and
\eqref{eq:quant-moment-bounds}, we deduce that
\begin{equation}\label{eq:quant-spherical-row}
 |2r^2A_r(x')-1+q\cdot \PP_rh(x')|
 \le4\sqrt{\mathcal{D}(\mu)}+10\mathcal{D}(\mu)\quad\text{for }(x',q)\in\supp(dx'\otimes\kappa_{x'}).
\end{equation}
Moreover, \eqref{avg.norm.to.zero} is now replaced by
$$\limsup_{L\to\infty}|h_{L,r}(0)|=\limsup_{R\to\infty}\biggl\lvert\int_K q\,d\beta_R\biggr\rvert\le2\sqrt{\mathcal{D}(\mu)},$$
thanks to \eqref{eq:quant-barycenter}
(as before, factors such as $e^{-|q-q_0|^2/(4R^2)}$ disappear in the limit). Thus, we obtain
\begin{equation}\label{eq:quant-root-mean}
 |\PP_rh(0)|\le2\sqrt{\mathcal{D}(\mu)}.
\end{equation}
Since $\hat x_0=(0,\hat q_0)\in\supp\eta_L$ for every $L$ (and \eqref{eq:quant-spherical-row} holds also at limits of points in $\supp\eta_L$), using \eqref{eq:quant-spherical-row} and \eqref{eq:quant-root-mean} at this point yields
\begin{equation}\label{eq:quant-root-defect}
 \biggl\lvert A_r(0)-\frac{1}{2r^2}\biggr\rvert
 \le\frac{3\sqrt{\mathcal{D}(\mu)}+5\mathcal{D}(\mu)}{r^2}.
\end{equation}

Now assume by contradiction that the bound stated in Theorem~\ref{thm:width} is false.
We can then find measures $\mu_j$, normalized as above (in particular,
such that $\bar B_1(0)$ is the smallest ball including $\pi_{V^\perp}\supp\mu_j$), with
$\mathcal{D}(\mu_j)\to0$. For each of them, we construct functions $A_r^{(j)},h^{(j)}$ and measures $\kappa_{x'}^{(j)}$
as above (using a sequence of scales $L\to\infty$ which may depend on $j$).

Since the measures $dx'\otimes\kappa_{x'}^{(j)}$ are locally uniformly bounded, after a subsequence
they converge to a limit $dx'\otimes\kappa_{x'}$. Also, since $|h^{(j)}|\le1$, we can extract a limit $h$
(with respect to weak-* convergence in $L^\infty$), so that $\PP_rh^{(j)}\to\PP_rh$ smoothly,
and similarly we can take limits $A_r^{(j)}\to A_r$ locally smoothly in $(x',r)$, thanks to the bound \eqref{eq:quant-root-defect} and \eqref{a.lb}. By \eqref{a.lb} again, in the limit $j\to\infty$ we have
$$A_r\ge0,\quad \PP_sA_r=A_{\sqrt{s^2+r^2}},$$
while \eqref{eq:quant-root-mean} and \eqref{eq:quant-root-defect} give
$$\PP_rh(0)=0,\quad A_r(0)=\frac{1}{2r^2},$$
and \eqref{eq:quant-spherical-row} becomes
$$2r^2A_r(x')=1-q\cdot \PP_rh(x')\quad\text{for }(x',q)\in\supp(dx'\otimes\kappa_{x'}).$$

Finally, since $dx'\otimes d\kappa_{x'}^{(j)}$ is supported in $\R^m\times[\bar B_1^{n-m}(0)\setminus B_{1-\delta_j}^{n-m}(0)]$ by \eqref{eq:quant-shell}, where $\delta_j:=4\sqrt{\mathcal{D}(\mu_j)}+8\mathcal{D}(\mu_j)\to0$,
we see that the limit measure $dx'\otimes\kappa_{x'}$ is supported in $\R^m\times\partial B_1^{n-m}(0)$,
so that $\kappa_{x'}$ is supported on the unit sphere $\mathbb{S}:=\partial B_1^{n-m}(0)$ (up to negligible modification).
Also, given any $\psi\in C^0_c(\R^m)$, we have
$$\int_{\R^m}h\psi\,dx'=\lim_{j\to\infty}\int_{\R^m}h^{(j)}\psi\,dx'
=\lim_{j\to\infty}\int_{\R^n}q\psi(x')\,dx'\otimes d\kappa_{x'}^{(j)}(q)
=\int_{\R^n}q\psi(x')\,dx'\otimes d\kappa_{x'}(q),$$
proving that $h(x')=\int_{\mathbb{S}}q\,d\kappa_{x'}(q)$, i.e., $h(x')$ is the center of mass of $\kappa_{x'}$,
for a.e.\ $x'$.

We are thus under the assumptions of Corollary~\ref{lem:fiber} with
$M=1$; since the latter forces $M=0$, we obtain a contradiction.

\section{Appendix: a rectifiability criterion}

The following is a simpler, weaker version of \cite[Lemma~5.2]{Preiss} and \cite[Theorem~5.3]{Preiss}. Its proof adapts
arguments from the rectifiability criterion of Marstrand--Mattila, as presented in \cite[Chapter 5]{DeLellis}
and \cite[Chapter~16]{MattilaBook}.

\begin{proposition}\label{thm:Cprime}
Given integers $1\le m\le n-1$, there exists a constant $c(m)>0$ such that the following holds.
Let $\mu$ be a positive Radon measure on $\R^n$, let $G\subseteq F$ be
compact sets, with $G$ purely unrectifiable, and suppose that
\begin{equation}\label{eq:growth}
 a r^m\le\mu(B_r(y))\le b r^m
 \quad\text{for }y\in F,\ r\in(0,r_0),
\end{equation}
for uniform parameters $0<a\le b$ and $r_0>0$.
Suppose also that, for every $x\in G$ and $r\in(0,r_0)$, there exists an $m$-plane $V=V(x,r)$ such that
\begin{gather}
 F\cap B_{2r}(x)\subseteq B_{\eta r}(V+x),\label{eq:tube}\\
 (V+x)\cap B_r(x)\subseteq B_{\eta r}(F),\label{eq:filling}
\end{gather}
with $\eta<c\cdot(a/b)$. Then we have $\mu(G)=0$.
\end{proposition}

Let us assume without loss of generality that $a=1$.

\begin{lemma}\label{lem:projection}
Given any $m$-plane $V$, for every $x\in G$ and $\sigma\in(0,r_0/3)$ we have
\begin{equation}\label{eq:projection}
 \HH^m(\pi_V[G\cap\overline B_\sigma(x)])
 \le C(m)\eta\cdot \mu(B_{2\sigma}(x)).
\end{equation}
\end{lemma}

\begin{proof}
Assuming without loss of generality that $\eta<1/100$,
note that for $\HH^m$-a.e.\ $y\in G\cap\overline B_\sigma(x)$ there exists $z\in G\cap B_\sigma(y)\setminus\{y\}$ such that
\begin{equation}\label{eq:vertical-chord}
 |\pi_V(z-y)|<\eta|z-y|.
\end{equation}
Indeed, the Borel set $S\subseteq G\cap\overline B_\sigma(x)$ of points $y$ where this fails cannot have $\HH^m(S)>0$, since otherwise we would have $|\pi_V(z-y)|\ge\eta|z-y|$ whenever $y,z\in S\cap B$, for any ball $B=B_{\sigma/2}(p)\subset\R^n$, implying that $S$ is included in a finite union of Lipschitz graphs over $V$; this contradicts the assumption that $G$ is purely unrectifiable.

Given any such $y\in G\cap\overline B_\sigma(x)$, taking $z\in B_{\sigma}(y)$ as in \eqref{eq:vertical-chord}
and setting $r_y:=|z-y|<\sigma$, we now apply \eqref{eq:tube} at $y$, with radius $r:=3r_y<r_0$. We obtain an $m$-plane $W$ such that $F\cap B_{6r_y}(y)\subseteq B_{3\eta r_y}(W+y)$ and thus, since $z\in G\subseteq F$,
there exists a unit vector $w\in W$ (approximating $(z-y)/r_y$) with $|\pi_Vw|\le C\eta$, for an absolute constant $C$.

In particular, the least singular value of $\pi_V|_W$ is at most $C\eta$.
Thus, considering the disk $D_y:=(W+y)\cap B_{5r_y}(y)$, the projection $\pi_V [B_{3\eta r_y}(D_y)]$
is included in a rectangle of diameter at most $C(m)r_y$, with one side of length at most $C\eta r_y$. Since
\[
 G\cap B_{5r_y}(y)\subseteq B_{3\eta r_y}(W+y)\cap B_{5r_y}(y)\subseteq B_{3\eta r_y}(D_y),
\]
recalling that $r_y^m\le\mu(B_{r_y}(y))$ (by \eqref{eq:growth} and $a=1$), we then get
\begin{equation}\label{eq:slab-volume}
 \HH^m(\pi_V[G\cap B_{5r_y}(y)])\le C(m)\eta r_y^m\le C(m)\eta\mu(B_{r_y}(y)).
\end{equation}
The claim follows by applying Vitali's covering lemma to the balls $\{B_{r_y}(y)\}$, as $y$
varies among the points such that \eqref{eq:vertical-chord} holds for some $z$ (its complement projects to a negligible set).
\end{proof}

\begin{proof}[Proof of Proposition~\ref{thm:Cprime}]
Assume by contradiction that $\mu(G)>0$. Let $x\in G$ be a $\mu$-density point for $G$;
up to a translation, we will assume that $x=0$ and write balls $B_r(0)$ simply as $B_r$.
The upper bound in \eqref{eq:growth} gives
\begin{equation}\label{eq:density-loss}
 \mu(B_{\sigma}\setminus G)=o(\sigma^m).
\end{equation}
We can then take $\sigma<r_0/3$ such that $\mu(B_{\sigma}\setminus G)<\min\{\eta^m,1/(4\cdot80^m)\}\sigma^m$.

Next, let $V:=V(0,\sigma)$ and, for any $v\in V$, let $D_r(v):=V\cap B_r(v)$ be the disk of center $v$ and radius $r$ in $V$. By Lemma~\ref{lem:projection}, the sets
\[
 K:=\pi_V(G\cap\overline B_\sigma),\quad D:=D_{\sigma/4}(0)
\]
contain $0$ and satisfy
\begin{equation}\label{eq:many-holes}
 \HH^m(D\setminus K)\ge\frac{\omega_m\sigma^m}{4^m}-C(m)\eta b\cdot (2\sigma)^m\ge\frac{\omega_m\sigma^m}{2\cdot 4^m},
\end{equation}
once we take $c(m)\le\omega_m/(C(m)2^{3m+1})$ in the statement.
Moreover, letting $$\rho:D\setminus K\to(0,\sigma/4),\quad\rho(v):=\dist(v,K),$$ we claim that
\begin{equation}\label{eq:hole-radius}
 \rho<2\eta\sigma.
\end{equation}
Indeed, for each $v\in D\setminus K$, \eqref{eq:filling} (with $x:=0$, $r:=\sigma$, and the plane $V$) yields some
$p\in F$ with $|p-v|<\eta\sigma$. Thus, $B_{\eta\sigma}(p)\subset B_\sigma$ projects into
$D_{2\eta\sigma}(v)$, which cannot be disjoint from $K$, as this would imply $B_{\eta\sigma}(p)\cap G=\emptyset$ and thus
$$(\eta\sigma)^m\le\mu(B_{\eta\sigma}(p))\le\mu(B_{\sigma}\setminus G),$$
contradicting our choice of $\sigma$. Since we just proved that $D_{2\eta\sigma}(v)$ must intersect $K$, \eqref{eq:hole-radius} follows.

Applying Vitali's covering lemma to $\{D_{2\rho(v)}(v)\mid v\in D\setminus K\}$, we get a disjoint family
$\{D_{2\rho_i}(v_i)\}_{i\in I}$ such that $D\setminus K\subseteq\bigcup_{i\in I}D_{10\rho_i}(v_i)$. Thus, by
\eqref{eq:many-holes}, we have
\begin{equation}\label{eq:total-holes}
 \sum_{i\in I}\rho_i^m\ge\frac{\sigma^m}{2\cdot40^m}.
\end{equation}
Denoting by $\mathcal C_s(v):=\pi_V^{-1}(D_s(v))$ the cylinder above $D_s(v)$, we now let
\[
 I':=\{i\in I:F\cap B_{\sigma/2}\cap \mathcal{C}_{\rho_i/2}(v_i)=\emptyset\}.
\]

The indices $i\not\in I'$ have a small contribution to the sum in \eqref{eq:total-holes}:
indeed, given $i\not\in I'$, taking $p_i\in F\cap B_{\sigma/2}\cap \mathcal{C}_{\rho_i/2}(v_i)$, we have
$B_{\rho_i/2}(p_i)\subseteq B_\sigma$. Since this ball projects into $D_{\rho_i}(v_i)$, it does not intersect $G$. Also, as $i\in I\setminus I'$ varies, these balls are disjoint, since the disks $D_{\rho_i}(v_i)$ are disjoint. We deduce that
\[
 \sum_{i\in I\setminus I'}\rho_i^m
 \le 2^m\sum_{i\in I\setminus I'}\mu(B_{\rho_i/2}(p_i))
 \le 2^m\mu(B_\sigma\setminus G)
 <\frac{\sigma^m}{4\cdot40^m},
\]
again thanks to our choice of $\sigma$. Consequently, recalling \eqref{eq:total-holes}, we have
\begin{equation}\label{eq:clean-holes}
 \sum_{i\in I'}\rho_i^m\ge\frac{\sigma^m}{4\cdot 40^m}.
\end{equation}

Fix now an index $i\in I'$ and choose $z_i\in G\cap\overline B_\sigma$ such that
$|\pi_Vz_i-v_i|=\rho_i$. By \eqref{eq:tube} we have
\[
 |z_i|\le|\pi_Vz_i|+\eta\sigma\le|v_i|+\rho_i+\eta\sigma\le(\sigma/4)+3\eta\sigma<\sigma/3.
\]
Let $W_i:=V(z_i,r_i)$, where $r_i:=\rho_i/(8\eta)<\sigma/4$.
The least singular value of $\pi_V|_{W_i}$ is at most $16\eta$: if not, then $w_i:=\pi_V|_{W_i}^{-1}(v_i-\pi_Vz_i)$ would have
\[
 \pi_V(z_i+w_i)=v_i,\quad
 |w_i|\le\frac{\rho_i}{16\eta}<\frac\sigma{8}.
\]
By \eqref{eq:filling} (at center $z_i$ and scale $r_i$), we could also find a point $p_i\in F\cap B_{\eta r_i}(z_i+w_i)$, thus in $F\cap B_{\sigma/2}\cap \mathcal{C}_{\rho_i/2}(v_i)$ (as it projects into $D_{\eta r_i}(v_i)$ and $\eta r_i=\rho_i/8<\eta\sigma/4$, so that $|p_i|\le\sigma/3+\sigma/8+\eta\sigma/4<\sigma/2$), contrary to the assumption $i\in I'$.
Hence, we can find a unit vector $\hat w_i\in W_i$ such that $|\pi_V\hat w_i|\le 16\eta$.

Let us take $N:=\lceil1/(32\eta)\rceil$ values $t_{i,1},\dots,t_{i,N}\in[0,r_i/4]$, with $|t_{i,k}-t_{i,k'}|\ge\rho_i$. For each $t_{i,k}$, \eqref{eq:filling} gives $q_{i,k}\in F$ with
\[
 |q_{i,k}-(z_i+t_{i,k}\hat w_i)|<\eta r_i=\rho_i/8.
\]
In particular, the balls $B_{\rho_i/4}(q_{i,k})$ are disjoint. Recalling that $|z_i|<\sigma/3$,
we have
$$|q_{i,k}|+\frac{\rho_i}{4}<|z_i|+|t_{i,k}|+\frac{\rho_i}{8}+\frac{\rho_i}{4}<\frac\sigma3+\frac{r_i}{4}+\rho_i<\frac\sigma2,$$
so that $B_{\rho_i/4}(q_{i,k})\subset B_{\sigma/2}$. Also, recalling that $|\pi_Vz_i-v_i|=\rho_i$, we have
$$|\pi_Vq_{i,k}-v_i|+\frac{\rho_i}{4}<|\pi_Vz_i-v_i|+16\eta|t_{i,k}|+\frac{\rho_i}{8}+\frac{\rho_i}{4}
<\frac{3\rho_i}{2}+16\eta\cdot\frac{r_i}{4}=2\rho_i,$$
so that $B_{\rho_i/4}(q_{i,k})$ projects into $D_{2\rho_i}(v_i)$.
Thus, we have
$$B_{\rho_i/4}(q_{i,k})\subseteq B_{\sigma/2}\cap \mathcal{C}_{2\rho_i}(v_i),$$
which means that also the family $\{B_{\rho_i/4}(q_{i,k})\}_{i\in I',k}$ is disjoint, giving
\[
\frac{\sigma^m}{256\cdot160^m\eta}\le\sum_{i\in I'}\frac{\rho_i^m}{64\eta\cdot 4^m}\le \sum_{i\in I'}N(\rho_i/4)^m
\le\sum_{i\in I'}\sum_{k=1}^N\mu(B_{\rho_i/4}(q_{i,k}))\le\mu(B_\sigma)\le b\sigma^m
\]
by \eqref{eq:clean-holes}, a contradiction once we choose $c(m)$ small enough in the statement.
\end{proof}

\end{document}